\documentclass{amsart}
\usepackage{amssymb}
\usepackage{hyperref}
\newtheorem{theorem}{Theorem}[section]

\newtheorem{fact}[theorem]{Fact}
\newtheorem{lemma}[theorem]{Lemma}

\newtheorem{claim}[theorem]{Claim}
\newtheorem{proposition}[theorem]{Proposition}

\newtheorem{introtheorem}{Theorem}
\newtheorem{introcon}[introtheorem]{Conjecture}

\theoremstyle{definition}
\newtheorem{definition}[theorem]{Definition}
\newtheorem{example}[theorem]{Example}
\newtheorem{non-example}[theorem]{Non-Example}
\newtheorem{convention}[theorem]{Convention}
\newtheorem{notation}[theorem]{Notation}

\newtheorem{assumption}[theorem]{Assumption}

\theoremstyle{remark}
\newtheorem{remark}[theorem]{Remark}

\newcommand{\rk}{\operatorname{rk}}
\newcommand{\acl}{\operatorname{acl}}

\newcommand{\cb}{\operatorname{Cb}}

\begin{document}

\thanks{Partially supported by the Fields Institute for Research in Mathematical Sciences; NSF grant DMS 1800692; ISF grant No. 555/21; a BGU Kreitman foundation fellowship; and a UMD Brin postdoc.}

\begin{abstract} We prove the higher dimensional case of the o-minimal variant of Zilber's Restricted Trichotomy Conjecture. More precisely, let $\mathcal R$ be an o-minimal expansion of a real closed field, let $M$ be an interpretable set in $\mathcal R$, and let $\mathcal M=(M,...)$ be a reduct of the induced structure on $M$. If $\mathcal M$ is strongly minimal and not locally modular, then $\dim_{\mathcal R}(M)=2$. As an application, we prove the Zilber trichotomy for all strongly minimal structures interpreted in  the theory of compact complex manifolds.
\end{abstract}


\title{The O-minimal Zilber Conjecture in Higher Dimensions}
\author{Benjamin Castle}
\address{Department of Mathematics, University of Maryland College Park}
\email{bcastle@berkeley.edu}
\maketitle

\section{Introduction}\label{S: intro}

This is the next paper in a series starting with \cite{CasACF0} and \cite{CaHaYe}. In \cite{CasACF0}, we proved the Zilber trichotomy for structures definable in ACF$_0$. In \cite{CaHaYe}, we gave an axiomatic framework for the techniques in \cite{CasACF0}, and deduced analogs for ACF$_p$ and ACVF. In this paper, we use the same axiomatic setting to study an o-minimal variant. Our main result is:

\begin{introtheorem}\label{T: intro} Let $\mathcal R$ be an o-minimal expansion of a real closed field. Let $\mathcal M=(M,...)$ be a non-locally modular strongly minimal $\mathcal R$-relic. Then $\dim_{\mathcal R}(M)=2$.
\end{introtheorem}

Recall (see \cite{CasHas}) that an $\mathcal R$-\textit{relic} is another structure whose definable sets are all interpretable sets in $\mathcal R$ (equivalently, an $\mathcal R$-interpretable structure equipped with a fixed interpretation in $\mathcal R$). Thus, Theorem \ref{T: intro} is a step toward proving the Zilber trichotomy for strongly minimal structures interpreted in o-minimal structures. As an application, we show that Theorem \ref{T: intro} implies the trichotomy for strongly minimal structures interpreted in compact complex manifolds (see Section \ref{S: CCM}). Below, we elaborate more on the meaning and significance of Theorem \ref{T: intro}.

\subsection{History and Peterzil's Conjecture} Recall that a structure $\mathcal M=(M,...)$ is \textit{strongly minimal} if every definable subset of $M$ is finite or cofinite (uniformly in families). A strongly minimal structure is \textit{locally modular} if it does not admit a two-dimensional family of plane curves (see Fact \ref{F: nlm via families}). Zilber famously conjectured that every non-locally modular strongly minimal structure interprets an algebraically closed field (this is Zilber's \textit{trichotomy conjecture}, see e.g. \cite[Theorem 3.1 an Conjecture B]{zil84}). Hrushovski found a family of counterexamples to Zilber's conjecture (\cite{Hrcon}). After Hrushovski's counterexamples, attention turned toward various restricted versions of the original conjecture, which turned out to be useful in applications (e.g. \cite{HrZil} and \cite{HrMordellLang}). 

Many papers (\cite{Martin}, \cite{Raz}, \cite{MaPi}, \cite{Ra}, \cite{KowRand}, \cite{HaSu}, \cite{CasACF0}, \cite{HaOnPi}, \cite{CaHaYe}) considered the trichotomy for strongly minimal relics of algebraically closed fields (and more generally algebraically closed valued fields). The case of pure ACF relics ultimately became known as the Restricted Trichotomy Conjecture, and was solved in \cite{CaHaYe}. Meanwhile, in 2005 Peterzil posed the trichotomy for strongly minimal relics of o-minimal structures (generalizing the Restricted Trichotomy in characteristic zero). Because all sufficiently rich o-minimal structures are locally expansions of real closed fields, it is reasonable to restrict Peterzil's question to the case of fields:\footnote{In fact, it is natural to wonder whether Peterzil's conjecture for expansion of fields directly implies Peterzil's conjecture for all o-minimal structures. We hope to prove this in a future paper.}

\begin{introcon}[Peterzil's Conjecture]\label{C: o-minimal conjecture} Suppose $\mathcal R$ is an o-minimal expansion of a real closed field, and $\mathcal M$ is a non-locally modular strongly minimal $\mathcal R$-relic. Then $\mathcal M$ interprets a copy of the algebraically closed field $R[i]$.
\end{introcon}

This is equivalent to $\mathcal M$ satisfying the Zilber trichotomy, because by \cite{PeSte} any infinite stable field interpreted in $\mathcal R$ is definably isomorphic to $R[i]$.

Peterzil's Conjecture is part of a larger program of classifying general relics of o-minimal structures. Many results along these lines were obtained by Hasson, Onshuus, and Peterzil (in different combinations -- see \cite{HaOn}, \cite{HaOn1}, \cite{HaOnPe}, \cite{HaOnPe1}). An imprecise and incomplete summary of the picture is the following: let $\mathcal M$ be any relic of an o-minimal structure. Then:

\begin{itemize}
    \item Types in $\operatorname{Th}(\mathcal M)$ admit finite coordinatizations into `rank 1' types.
    \item Every rank 1 type is (essentially) either o-minimal or of U-rank 1.
    \item A satisfying classification of the o-minimal case is given by the Peterzil-Starchenko trichotomy (\cite{PeStTricho}). 
    \item A satisfying classification of the U-rank 1 case would amount to solving Peterzil's conjecture.
\end{itemize}

In particular, a solution of Peterzil's conjecture would give a satisfying (local) description of \textit{all} relics of o-minimal structures (or at least relics of o-minimal expansions of fields). This is similar to the ACF-case, where the trichotomy for strongly minimal relics implies a similar statement for all relics (i.e. the main theorem of \cite{CasACF0}). 


\subsection{Past Results}

Suppose $\mathcal R$ and $\mathcal M$ are as in Peterzil's conjecture (so we assume $\mathcal M$ is not locally modular). Let $M$ be the universe of $\mathcal M$. By elimination of imaginaries in $\mathcal R$, we can assume $M\subset R^n$ for some $n$. Let $\dim_{\mathcal R}(M)$ be the o-minimal dimension of the set $M$. 

Conjecture \ref{C: o-minimal conjecture} naturally divides into cases according to the value $\dim_{\mathcal R}(M)$. This is similar to the original Restricted Trichotomy Conjecture, but a bit more complicated. Working over ACF, one considers the `one-dimensional' and `higher-dimensional' cases -- essentially because the desired interpreted field is one-dimensional. In the o-minimal setting, the desired interpreted field $R[i]$ is \textit{two-dimensional} -- and accordingly, there are \textit{three} natural cases for $\dim_{\mathcal R}(M)$: 1, 2, and $\geq 3$. Indeed, a proof of the conjecture is expected to comprise the following four steps:\\

\textbf{Step 1:} $\dim_{\mathcal R}(M)$ cannot be 1.

\textbf{Step 2:} $\dim_{\mathcal R}(M)$ cannot be $\geq 3$.

\textbf{Step 3:} If $\dim_{\mathcal R}(M)=2$, then interprets a strongly minimal group.

\textbf{Step 4:} If $\dim_{\mathcal R}(M)=2$ and $\mathcal M=(M,+,...)$ is an expansion of a group, then $\mathcal M$ interprets a copy of $R[i]$.\\

Steps 1 and 4 were done previously. Step 1 is due to Hasson, Onshuus, and Peterzil (\cite{HaOnPe}). Step 4 is due to Eleftheriou, Hasson, and Peterzil (\cite{ElHaPe}), building on work of Hasson and Kowalski (\cite{HaKo}).

Our Theorem \ref{T: intro} is precisely Step 2. Thus, we reduce Peterzil's Conjecture to Step 3 (interpreting a strongly minimal group in dimension 2). The ultimate goal is to prove this final step by combining our techniques with those of \cite{ElHaPe}. 



\subsection{Summary of Proof}

We now summarize the proof of Theorem \ref{T: intro}. Most of the proof will be done axiomatically, using language from \cite{CaHaYe} (with some additions). For now, for concreteness, we stick to the o-minimal setting.

The proof of Theorem \ref{T: intro} follows the same general strategy as \cite{CasACF0} and \cite{CaHaYe} (at least their `higher-dimensional' parts), but the o-minimal version introduces several challenges. We start by reviewing the strategy in \cite{CasACF0}. So let $\mathcal M=(M,...)$ be a non-locally modular strongly minimal $(\mathbb C,+,\cdot)$-relic (where the goal is to show that $\dim_{\mathbb C}(M)=1$). The following are the main steps of the argument:

\begin{enumerate}
    \item Prove a `closure detection' result, stating that $\mathcal M$ can partially recover the closure operator on $\mathcal M$-definable sets.
    \item Construct a sufficiently well-behaved \textit{ramification} point of an $\mathcal M$-definable map $f:X\rightarrow Y$ (i.e. a point in $X$ where the differential is not injective and with some nice geometric properties).
    \item Using the \textit{Purity of the Ramification Locus}, show that $\dim(\operatorname{Ram}(f))\geq\dim(X)-1$, where $\operatorname{Ram}(f)\subset X$ is the ramification locus of $f$.
    \item On the other hand, use (1) to show that $\operatorname{Ram}(f)$ is (approximately) $\mathcal M$-definable.
    \item By (3) and (4), conclude that $\mathcal M$ defines two sets whose complex dimensions differ by 1; use this to show $\dim(M)=1$.
\end{enumerate}

Step (1) is the hardest part of \cite{CasACF0}. Thankfully, this step was treated axiomatically in \cite{CaHaYe}, and we get it for free from the axiomatic version (see Fact \ref{F: closure}). Likewise, steps (4) and (5) will transfer appropriately when they are needed. The challenges are steps (2) and (3). In fact, (2) would also transfer smoothly as stated, but we will have to adapt (3) in a way that also changes (2). 

The issue with (3) is that purity of ramification is an algebraic fact about morphisms of smooth varieties, which is easily seen to fail in o-minimal structures. For example, the complex squaring map can be viewed as a semi-algebraic map with ramification locus of o-minimal codimension 2. Moreover, this is the only place in \cite{CasACF0} where the background geometry interacts with the dimension of $M$, so the strategy cannot be salvaged without this step.
 
As it turns out, there is an o-minimal version of the purity of the ramification locus, which we prove in an appendix (Proposition \ref{ramification prop}). This seems not to have been known, and may be of interest in its own right:

\begin{proposition}[Proposition \ref{ramification prop}] Let $\mathcal R=(R,+,\cdot,<,...)$ be an o-minimal expansion of a real closed field. Let $X$ and $Y$ be definable $n$-manifolds over $\mathcal R$, let $f:X\rightarrow Y$ be a continuous definable finite-to-one map, and let $\operatorname{Ram}(f)\subset X$ be the set of points where $f$ is not locally injective. Then either $\operatorname{Ram}(f)=\emptyset$, or $\dim(\operatorname{Ram}(f))=n-1$, or $\dim(\operatorname{Ram}(f))=n-2$.
\end{proposition} 

That is, the `complex squaring' example above is the worst-case scenario: the codimension of $\operatorname{Ram}(f)$ cannot go beyond 2 (unless $\operatorname{Ram}(f)$ is empty). Notice that the bound $\operatorname{codim}(\operatorname{Ram}(f))\leq 2$ is perfect for adapting step (5) above (since we want to show $\dim(M)\leq 2$ as opposed to $\dim(M)\leq 1$ in the algebraic setting).

An interesting feature of Proposition \ref{ramification prop} is that it is purely topological; in fact, the proof involves only o-minimal algebraic topology, and does not use any differential structure. This was surprising, as differential data typically plays a significant role in the Zilber trichotomy. In fact, because of the statement of Proposition \ref{ramification prop}, the entire proof of Theorem \ref{T: intro} (modulo the results of \cite{CaHaYe}) is done in the topological category.

Proposition \ref{ramification prop} might seem to fix everything; however, it also introduces a new problem. Notice that Proposition \ref{ramification prop} only works for \textit{finite-to-one} maps (we do not know if `generically finite-to-one' is enough). This actually bans the version of step (2) in \cite{CasACF0}. In that paper, (2) was done using a self-intersection of plane curves: given a family $\{C_t:t\in T\}$ of plane curves in $\mathcal M$, we chose a generic $\hat t\in T$, and a generic $\hat x\in C_{\hat t}$. Then the intersection $\hat x\in C_{\hat t}\cap C_{\hat t}$ is trivially a ramification point of an appropriate map, because it belongs to an infinite fiber. The same trick will not work in the current paper. Thus, the rest of the proof of Theorem \ref{T: intro} is a modified construction of a ramification point of a finite-to-one map.

First we assume, as in \cite{ElHaPe}, that $M$ carries a definable group structure (say $(M,+)$). In this case there is a nice trick to avoid infinite fibers: we work with the family $\{C+t:t\in M^2\}$ of translates of a distinguished plane curve $0\in C\subset M^2$; but instead of looking at pairwise intersections of the family $\{C+t\}$ with itself, we look at intersections between the two \textit{different} families $\{C+t\}$ and $\{-C+t\}$. This produces a similar ramification point at $0\in C\cap -C$ -- but because the intersection is no longer a self-intersection, we can force all nearby fibers to be finite. As it turns out, this argument can be expressed quite succinctly after simplifying with the group operation -- the shortest version is presented in Section \ref{group section}.

We then move to the general case. In Section \ref{S: informal} we informally discuss how to abstract a general strategy from the group case. The idea is to carry out the same construction, replacing the family of translates of $C$ with a rank 2 family $\{C_t:t\in T\}$ indexed by a set $T\subset M^2$; and replacing the `negative' operation $t\mapsto -C+t$ with the `dual' operation $t\mapsto\{x:t\in C_x\}$. In the case of groups, these two operations are equivalent, so this is a natural idea.

We then run the general argument in Section 6. The proof works assuming certain nice properties of the family $\{C_t:t\in T\}$ -- namely, that dualizing indeed changes the family (to avoid infinite intersections), and that $t\in C_t$ holds in a rather strong way for (almost) all $t\in T$ (this replaces the assumption that $0\in C$ in the group case). The remaining task is to construct a family $\{C_t:t\in T\}$ with these properties. We do this in Section \ref{construction section}. As it turns out, there is a construction that either works, or produces a strongly minimal group configuration. So in the end, we either run the general argument or reduce to the group case.

Finally, in Section \ref{S: main thm} we collect the pieces of the argument and conclude Theorem \ref{T: intro}; and in Section \ref{S: CCM}, we deduce the trichotomy for relics of compact complex manifolds (CCM). This is a rather immediate consequence of Theorem \ref{T: intro}. Let $\mathcal M=(M,...)$ be a non-locally modular strongly minimal CCM-relic. Since CCM is interpretable in the o-minimal structure $\mathbb R^{an}$, we can view $\mathcal M$ as an o-minimal relic. By Theorem \ref{T: intro}, $M$ has o-minimal dimension 2, and thus CCM dimension 1. But all one-dimensional CCM-definable sets are algebraic -- so we reduce to the trichotomy for ACF$_0$-relics.

\section{Preliminaries}

We assume basic model theory; see \cite{mar} for any necessary facts and definitions. Since they have been discussed at length in recent similar works, we also assume many basic properties of geometric and strongly minimal structures (e.g. the dimension theory of tuples and definable sets). For a survey, see the second sections of \cite{CaHaYe} and \cite{CasACF0}. We do remind and clarify a few things in Section \ref{S: setting}.

Throughout, we denote structures by calligraphic letters such as $\mathcal M$ and $\mathcal R$. The universe of a structure is denoted by the corresponding plain letter -- e.g. $M$ and $R$, respectively. Parameter sets are usually (but not always) denoted $A,B,...$, and tuples are typically denoted either $a,b,c,...$ or $x,y,z,...$. Throughout, the word `definable' allows parameters but not imaginaries. When dealing with imaginary sorts, we use `interpretable'. Finally, in the presence of more than one structure, we use the notation $\mathcal M(A)$-definable as a shorthand for `definable over $A$ in $\mathcal M$' (and similarly for interpretable).

\subsection{t-minimal Hausdorff Geometric Structures}

In \cite{CaHaYe}, we defined \textit{Hausdorff geometric structures} as an abstract setting for the Zilber trichotomy. In that paper, we notated such structures as $(\mathcal K,\tau)$. In this paper, we use $(\mathcal R,\tau)$ to match the standard notation for o-minimal structures.

\begin{notation} Throughout the paper, in the presence of a topology, we use $\overline X$ to denote the closure of $X$, and $\operatorname{Fr}(X)$ to denote the frontier of $X$ (i.e. $\overline X-X$).
\end{notation}

\begin{definition}\label{D: HGS}
    Let $\mathcal R$ be a structure, and let $\tau$ be a topology on $R$, extended to a topology on every definable subset of each $R^n$ using the product and subspace topologies. We call $(\mathcal R,\tau)$ a \textit{Hausdorff geometric structure} if the following hold:
    \begin{enumerate}
        \item $\tau$ is Hausdorff. 
    \item $\mathcal R$ is geometric and $\aleph_1$-saturated (but $\mathcal L$ might be uncountable).
    \item (Strong Frontier Inequality) If $X\subset K^n$ is definable over $A$ and 
    $a\in\overline{\operatorname{Fr}(X)}$ then $\dim(a/A)<\dim(X)$.
    \item (Baire Category Axiom) Let $X\subset R^n$ be definable over a countable set $A$, and let $a\in X$ be generic over $A$. Let $B\supset A$ be countable. Then every neighborhood of $a$ contains a generic of $X$ over $B$.
    \item (Generic Local Homeomorphism Property) Suppose $X$, $Y$, and $Z\subset X\times Y$ are definable over $A$ of the same dimension, and each of $Z\rightarrow X$ and $Z\rightarrow Y$ is finite-to-one. Let $(x,y)$ be generic in $Z$ over $A$. Then there are open neighborhoods $U$ of $x$ in $X$, and $V$ of $y$ in $Y$, such that the restriction of $Z$ to $U\times V$ is the graph of a homeomorphism $U\rightarrow V$.   
    \end{enumerate}
\end{definition}

In \cite{CaHaYe}, it was important that the topology $\tau$ need not admit a definable basis. In this paper, everything is modeled after densely ordered structures -- so there seems to be little harm in assuming a definable basis. In this case, there are only finitely many isolated points of $R$ -- so it is harmless to remove them. In other words, in addition to Definition \ref{D: HGS}, we assume $(\mathcal R,\tau)$ is \textit{t-minimal} (\cite{mathews}, Definition 2.4).

\begin{definition}\label{L: definable HGS} A \textit{t-minimal Hausdorff geometric structure} is a Hausdorff geometric structure with a $\emptyset$-definable basis and no isolated points; where by `$\emptyset$-definable basis', we mean a formula $\phi(x,\overline y)$ without parameters so that the instances $\phi(x,\overline a)$ for $\overline a\in R^{|\overline x|}$ form a basis of open sets for the topology $\tau$.
\end{definition}

\begin{remark}
    Among the \textit{uniform structures}, t-minimal Hausdorff geometric structures are precisely the visceral structures with the exchange property (see \cite{visceral} and \cite{johvisceral} for more about these notions). We do not prove this (and do not need it).
\end{remark}

\begin{remark}
    For t-minimal Hausdorff geometric structures, there is no need to assume $\aleph_1$-saturation in the definition (i.e. given a definable basis, one can rewrite the definition in a way that doesn't depend on the specific model). We still assume $\aleph_1$-saturation for consistency with \cite{CaHaYe}.
\end{remark}

For t-minimal structures, Definition \ref{D: HGS} simplifies considerably. We will not need this, so we omit the proof.

\begin{lemma} Let $\mathcal R$ be an $\aleph_1$-saturated geometric structure. Let $\tau$ be a Hausdorff topology on $R$ with a $\emptyset$-definable basis and no isolated points. Then $(\mathcal R,\tau)$ is a t-minimal Hausdorff geometric structure if and only if the following hold:
\begin{enumerate}
    \item For every definable set $X$, $\dim(\operatorname{Fr}(X))<\dim(X)$.
    \item If $X$ and $Y$ are definable, and $f:X\rightarrow Y$ is a definable function, then there is a definable $X'\subset X$ such that $f$ is continuous on $X'$ and $\dim(X-X')<\dim(X)$.
\end{enumerate}
\end{lemma}

\subsection{Enough Open Maps}

Suppose $(\mathcal R,\tau)$ is a Hausdorff geometric structure. In \cite{CaHaYe}, we also defined the assertion `$(\mathcal R,\tau)$ has \textit{enough open maps}' \cite[Definition 3.25]{CaHaYe}. This is a technical condition that we don't state, but is needed for most of the results of \cite{CaHaYe}.

\begin{fact}\label{F: enough opens}\cite[Theorem 9.4 and Proposition 3.40]{CaHaYe}  Every $\aleph_1$-saturated o-minimal expansion of a real closed field (equipped with the order topology) is a t-minimal Hausdorff geometric structure with enough open maps.
\end{fact}

\subsection{Manifolds in t-minimal Hausdorff Geometric Structures}\label{SS: Man}

Many of the results of \cite{CaHaYe} relied on an abstract definition of `smoothness'. Given a Hausdorff geometric structure $(\mathcal R,\tau)$, a notion of smoothness is a map $X\mapsto X^S$ assigning each definable set to a `smooth locus' and satisfying certain properties.

O-minimal expansions of fields carry a canonical notion of smoothness, coming from o-minimal differential geometry. We expected to use this extensively, but ultimately it will not be used at all (aside from quoting results from \cite{CaHaYe}). Instead, the pure topological category is more effective. Thus, we now introduce \textit{manifolds} as a `purely topological' analog of smoothness. 

In fact, we do something a bit more general. First, we define `lores' (short for \textbf{lo}cal \textbf{re}ducts) of a t-minimal Hausdorff geometric structure: a lore $\mathcal S$ is a certain rich-enough collection of definable sets (called loric sets) that are determined locally and satisfy certain preservation properties (approximating the notion of a reduct of $\mathcal R$). The canonical lore (and the one to use in the o-minimal setting) is the \textit{full lore} (containing every definable set).

Now suppose $(\mathcal R,\tau)$ is a t-minimal Hausdorff geometric structure with a lore $\mathcal S$. We then define the class of loric manifolds (i.e. manifolds with respect to $\mathcal S$): these are just like usual manifolds, except the charts are loric maps (i.e. maps with loric graph). In the o-minimal setting, if $\mathcal S$ is the full lore, the loric manifolds are the usual definable manifolds.

Our motivation for this approach comes from real closed valued fields (RCVF). In a future paper, we plan to extend our results to certain expansions of RCVF. However, since valued fields are totally discontinuous, the pure topological category won't be as useful. Instead, we will work with loric maps and loric manifolds, with loric interpreted as `locally semialgebraic' in an appropriate sense. For subtle reasons, some of our results are true for locally semialgebraic sets and false in general. We delay further discussion of this issue to a future paper.

\begin{definition}\label{D: strong set}
    Let $(\mathcal R,\tau)$ be a t-minimal Hausdorff geometric structure. A \textit{lore} of $(\mathcal R,\tau)$ is a collection $\mathcal S$ of $\mathcal R$-definable sets (called \textit{loric sets}) such that:
    \begin{enumerate}
        \item\label{D: strong sets basic} (Basic Properties) $R$ is loric. The diagonal $\{(x,x)\}\subset R^2$ is loric. The loric sets are closed under finite products, finite intersections, and coordinate permutations.
        \item\label{D: strong sets invariance} (Invariance) If $f:X\rightarrow Y$ is a definable homeomorphism, $X$ is loric , and the graph of $f$ is loric, then $Y$ is loric.
        \item\label{D: strong sets loc} (Locality) Every definable open subset of a loric set is loric. If $X$ is definable, and $\{U_i:i\in I\}$ is an open cover of $X$ by loric sets, then $X$ is loric.
        \item\label{D: strong sets gen} (Genericity) If $X\neq\emptyset$ is $A$-definable, then there is a relatively open $A$-definable $X'\subset X$ such that $\dim(X-X')<\dim(X)$ and $X'$ is loric.
    \end{enumerate}
     Suppose $\mathcal S$ is a local reduct. An \textit{loric map} is a continuous function whose graph is loric. A \textit{loric homeomorphism} is loric map that is also a homeomorphism. Two sets $X,Y$ are \textit{lorically homeomorphic} if there is a loric homeomorphism $f:X\rightarrow Y$.
\end{definition}

\begin{remark} Definition \ref{D: strong set}(1) and (2) approximate the assertion that $\mathcal S$ contains the definable sets of some reduct of $\mathcal R$ (where (2) is the analog of existential quantification). Of course, an actual reduct would entail closure under complements, unions, and projections -- but these are less compatible with locality, and can be seen to fail for locally semialgebraic sets in RCVF (Example \ref{E: RCF Lore}).
\end{remark}

\begin{remark}
    The term `loric' depends on the lore $\mathcal S$, and should really be `loric for $\mathcal S$'. Since we will work in a t-minimal Hausdorff geometric structure with a fixed lore, we stick to `loric'. This also applies to loric manifolds (defined below).
    \end{remark}

Several other properties follow easily from (1)-(4) above. We omit the details.

\begin{lemma}\label{L: strong basic} Suppose $(\mathcal R,\tau,\mathcal S)$ is a t-minimal Hausdorff geometric structure, and $\mathcal S$ is a lore of $(\mathcal R,\tau)$. 
\begin{enumerate}
    \item\label{L: finite sets strong} Finite sets (in $R^n$) are loric.
    \item\label{L: strong projection property} If $X$ and $Y$ are loric, then the projections $X\times Y\rightarrow X$ and $X\times Y\rightarrow Y$ are loric maps.
    \item\label{L: domain is strong} (Preimage Property) If $f:X\rightarrow R^n$ is loric, and $Y\subset R^n$ is loric, then $f^{-1}(Y)$ is loric. In particular, the domain of any loric map is loric.
    \item\label{L: strong inverse} (Inverse property) The inverse of a loric homeomorphism is loric.
    \item\label{L: strong coordinate} (Coordinate-by-coordinate property) Suppose $f_i:X\rightarrow R$ is loric for $i=1,...,n$, and $f:X\rightarrow R^n$ is given by $f_i$ on each coordinate. Then $f$ is loric if and only if each $f_i$ is.
    \item\label{L: strong homeo strong} If $X$ and $Y$ are lorically homeomorphic, then $X$ and $Y$ are loric.
    \item\label{L: strong composition} (Composition property) The composition of two loric maps is loric, and the identity map on any loric set is loric (thus the relation `$X$ and $Y$ are strongly homeomorphic' is an equivalence relation on loric sets).
    \item\label{L: strong rest} (Restriction property) If $f:X\rightarrow Y$ is loric, and $X'\subset X$ is loric, then $f\restriction_{X'}$ is loric.
\end{enumerate}
\end{lemma}

Note that the image of a loric map need not be loric (this only holds for homeomorphisms). In particular, while a projection of loric sets is a loric map, the projection of a loric set need not be a loric set.

\begin{example}
    Let $(\mathcal R,\tau)$ be any t-minimal Hausdorff geometric structure. Let $\mathcal S$ be the collection of all $\mathcal R$-definable sets. Then $\mathcal S$ is a lore of $(\mathcal R,\tau)$. We call $\mathcal S$ the \textit{full lore}.
\end{example}

\begin{example}\label{E: RCF Lore}
    Let $(\mathcal R,\tau)$ be a real closed valued field. Say that a definable set $X$ is \textit{locally semialgebraic} if it admits an open cover $\{U_i:i\in I\}$ such that each $U_i$ is definable in the pure field language. Let $\mathcal S$ be the collection of all locally semialgebraic sets. Then $\mathcal S$ is a lore of $(\mathcal R,\tau)$. We leave the details to a future paper.
\end{example}

\begin{definition}\label{D: manifold}
    Let $(\mathcal R,\tau)$ be a t-minimal Hausdorff geometric structure, and let $\mathcal S$ be a lore of $(\mathcal R,\tau)$. Let $X$ be definable with $\dim(X)=n$, and let $x\in X$. 
    \begin{enumerate}
        \item $x$ is a \textit{loric manifold point of $X$} if there is a definable set $U$ with $x\in U\subset X$, $U$ open in $X$, and $U$ is lorically homeomorphic to an open subset of $R^n$.
        \item The set of loric manifold points of $X$ is called the \textit{loric manifold locus of $X$} and is denoted $X^{lman}$.
        \item A \textit{loric $n$-manifold} is a definable set $Y$ with $\dim(Y)=n$ and $Y^{lman}=Y$.
    \end{enumerate}
    \end{definition}

    Note that a loric manifold is loric, because it admits an open cover by loric sets (Definition \ref{D: strong set}(\ref{D: strong sets loc})).

\begin{example}
    Let $(\mathcal R,\tau,\mathcal S)$ be an $\aleph_1$-saturated o-minimal expansion of a field with the full lore. Then the loric $n$-manifolds are precisely the definable sets which are definable $n$-manifolds in the usual sense. This is not automatic, because in o-minimality one requires a finite atlas. However, for definable sets, any atlas implies the existence of a finite atlas (\cite[Proposition 4.2]{BerOte}).
\end{example}

In general, the loric manifold locus need not be definable (because one would need to quantify over all potential loric homeomorphisms). It \textit{is} definable in o-minimal expansions of fields with the full lore, but we won't need this.

We now give a few basic properties of loric manifolds. \textbf{For the rest of Subsection \ref{SS: Man}, fix $(\mathcal R,\tau,\mathcal S)$, a t-minimal Hausdorff geometric structure with a lore.}

\begin{lemma}[Generic Local Loric Homeomorphism Property]\label{L: strong local homeo} Let $X$ and $Y$ be $A$-definable sets of the same dimension, and let $f:X\rightarrow Y$ be a finite-to-one $A$-definable function. Let $x\in X$ be generic over $A$. Then $f$ restricts to a loric homeomorphism between neighborhoods of $x$ and $f(x)$.
\end{lemma}
\begin{proof} By the genericity property of loric sets (Definition \ref{D: strong set}(\ref{D: strong sets gen})), there is an $A$-definable open $X'\subset X$ with $\dim(X-X')<\dim(X)$ such that the restriction $f\restriction_{X'}$ is loric. Then $x$ is generic in $X'$ over $A$. By the generic local homeomorphism property, $f{\restriction_{X'}}$ restricts further to a homeomorphism $f\restriction_U:U\rightarrow V$ between neighborhoods of $U$ of $x$ and $V$ of $f(x)$ (in $X'$ an $Y$, respectively). By locality (Definition \ref{D: strong set}(\ref{D: strong sets loc})), $U$ is loric. By the restriction property (Definition \ref{L: strong basic}(\ref{L: strong rest})), $f\restriction_U$ is loric.
\end{proof}

 \begin{lemma}[Genericity Property of Loric Manifolds]\label{L: manifold genericity}
        Let $X$ be definable over $A$, and let $x\in X$ be generic over $A$. Then $x\in X^{lman}$.
    \end{lemma}
    \begin{proof} Let $n=\dim(X)$. Using compactness, one can construct an $A$-definable finite-to-one map $f:X\rightarrow R^n$. Now apply Lemma \ref{L: strong local homeo}.
    \end{proof}

    \begin{lemma}[Locality Property of Loric Manifolds]\label{L: manifold locality} Let $X$ be a definable set. Suppose for some $n$ that $X$ admits an open cover by loric $n$-manifolds. Then $n=\dim(X)$, and so $X$ is a loric $n$-manifold.
    \end{lemma}
    \begin{proof}
        Let $d=\dim(X)$. Let $x\in X$ be generic over any set of parameters defining $X$. By Lemma \ref{L: manifold genericity} applied to $x$, there is a definable, relatively open $Y_1\subset X$ containing $x$ and in definable bijection with an open subset of $M^d$. By assumption, there is $Y_2$ satisfying the same property with $d$ replaced by $n$. Shrinking if necessary, we may assume $Y_1=Y_2$. Thus, we get a definable bijection between open subsets of $M^n$ and $M^d$. 

        So to prove the lemma, it suffices to show that whenever $U\subset M^k$ is non-empty, definable, and open (for any $k$), we have $\dim(U)=k$. Indeed, it is clear that $\dim(U)\leq k$ because $U\subset M^k$. Conversely, since $U$ is open, it contains a box $B_1\times...\times B_k$, where each $B_i$ is non-empty and open in $R$. Since $R$ has no isolated points, each $B_i$ is infinite, so each $\dim(B_i)=1$, and thus $\dim(U)\geq\dim(B_1\times...\times B_k)=k$.
    \end{proof}

    \begin{lemma}[Preservation Properties of Loric Manifolds]\label{L: manifold preservation}\hfill
  \begin{enumerate}
        \item If $X$ is a loric $m$-manifold, and $Y$ is a loric $n$-manifold, then $X\times Y$ is a loric $m+n$-manifold.
         \item If $X$ and $Y$ are lorically homeomorphic, and $X$ is a loric $n$-manifold, then so is $Y$.
         \item Any definable open subset of a loric $n$-manifold is a loric $n$-manifold.
    \end{enumerate}
    \end{lemma}
    \begin{proof} (1) uses the coordinate-by-coordinate property of loric maps (Lemma \ref{L: strong basic}(\ref{L: strong coordinate})). (2) uses the composition property of loric maps (Lemma \ref{L: strong basic}(\ref{L: strong composition})). (3) uses Lemma \ref{L: manifold locality} and the restriction property of loric maps (Lemma \ref{L: strong basic}(\ref{L: strong rest})).
    \end{proof}        

\subsection{Purity of Ramification}

The analog of Theorem \ref{T: intro} in \cite{CaHaYe} relied on an axiomatic statement of the purity of the ramification locus from algebraic geometry (\cite[Definition 6.16]{CaHaYe}). That definition fails in o-minimal structures. Instead, we introduce a variant involving loric maps of loric manifolds. The following is modeled after Proposition \ref{ramification prop} (our o-minimal purity of ramification).

\begin{definition}
    $f:X\rightarrow Y$ be a map of Hausdorff topological spaces. We say that $x\in X$ is a \textit{ramification point of $f$}, or that $f$ \textit{ramifies at $x$}, if $f$ is not injective on any neighborhood of $x$. The set of ramification points of $f$ is called the \textit{ramification locus of $f$}, and denoted $\operatorname{Ram}(f)$.
\end{definition}

Note that if $X$, $Y$, and $f$ are definable in a t-minimal Hausdorff geometric structure, then $\operatorname{Ram}(f)$ is definable without additional parameters.

\begin{definition}\label{D: ramification purity}
    Let $(\mathcal R,\tau,\mathcal S)$ be a t-minimal Hausdorff geometric structure with a fixed lore. Let $d\in\mathbb N$. We say that $(\mathcal R,\tau,\mathcal S)$ has \textit{manifold ramification purity of order $d$} if the following holds: let $X$ and $Y$ be loric $n$-manifolds (for the same $n$), and let $f:X\rightarrow Y$ be a finite-to-one loric map. Then either $\operatorname{Ram}(f)=\emptyset$ or $\dim(\operatorname{Ram}(f))\geq n-d$. 
\end{definition}

Later (Proposition \ref{ramification prop}) we will show that o-minimal expansions of real closed fields  (with the full lore) have manifold ramification purity of order 2. That is, for a continuous, definable, finite-to-one map of definable manifolds of the same dimension, the ramification locus is either empty, of codimension 1, or of codimension 2.

In a future paper, we will point out that RCVF has manifold ramification purity of order 2 when $\mathcal S$ consists of the locally semialgebraic sets, but \textit{not} when $\mathcal S$ is the full lore. This is the main reason for introducing lores.

We expect that for expansions of fields, there are no non-trivial examples of manifold ramification purity of order greater than 2. 

\section{The Setting}\label{S: setting}

\begin{assumption}\label{A: R and M} \textbf{From now until the end of Section \ref{construction section}, we fix the following:}

\begin{enumerate}
    \item \textbf{$(\mathcal R,\tau,\mathcal S)$, a t-minimal Hausdorff geometric structure with enough open maps and a fixed lore, which we assume to satisfy manifold ramification purity of order $d_0$ for some fixed $d_0\in\mathbb N$.} For readability, there is no harm in assuming that $(\mathcal R,\tau,\mathcal S)$ is an o-minimal expansion of a real closed field with the full lore (in which case $d_0=2$).
    \item \textbf{$\mathcal M$, a non-locally modular strongly minimal definable $\mathcal R$-relic.} By `definable', we mean that $M\subset R^n$ for some $n$. If $\mathcal R$ eliminates imaginaries (in particular if $\mathcal R$ is an o-minimal expansion of a field), this is a harmless assumption: in that case, every $\mathcal R$-relic is $\mathcal R$-definably isomorphic to a definable $\mathcal R$-relic.
    \item \textbf{We assume that the language of $\mathcal M$ is countable, that $\mathcal M$ has infinite $\acl(\emptyset)$, and that every $\emptyset$-definable set in $\mathcal M$ is also $\emptyset$-definable in $\mathcal R$.} These requirements can always be arranged: first pass to a finite sublanguage witnessing non-local modularity; then satisfy the other two requirements by naming countably many constants. 
    \end{enumerate}
\end{assumption}

 \subsection{Dimension and Rank}

 We follow the various notational conventions in \cite{CasACF0} (Section 3) to distinguish between $\mathcal R$ and $\mathcal M$. So \textit{dimension} and dim will always refer to the geometric (i.e. acl) dimension in $\mathcal R$, and \textit{rank} and rk refer to Morley rank in $\mathcal M$. \textbf{Unless stated otherwise, all tuples are assumed to come from $\mathcal M^{\textrm{eq}}$, and all parameter sets are countable subsets of $M$.} Thus, by $\aleph_1$-saturation, definable sets always have generic points over parameter sets.
 
Fact \ref{F: dim rk} is \cite[Lemma 4.11]{CaHaYe}:
 
 \begin{fact}\label{F: dim rk} Let $X\subset\mathcal M^{\textrm{eq}}$ be $\mathcal M(A)$-definable, and let $a\in X$.
 \begin{enumerate}
     \item $\dim X=\dim M\cdot\rk X$.
     \item $\dim(a/A)\leq\dim M\cdot\rk(a/A)$.
     \item If $a$ is generic in $X$ over $A$, then $a$ is $\mathcal M$-generic in $X$ over $a$.
 \end{enumerate}
 \end{fact}
 Given $\mathcal R$-definable sets $X\subset Y$, we say that $X$ is \textit{generic} in $Y$ if $\dim(X)=\dim(Y)$, and \textit{large} in $Y$ if $\dim(Y-X)<\dim Y$. We also say that two $\mathcal R$-definable sets are \textit{almost equal} if their intersection is large in each; and we say that $X$ is \textit{almost contained in} $Y$ if $\dim(X-Y)<\dim X$. If $X$ and $Y$ are in addition $\mathcal M$-definable, Fact \ref{F: dim rk} says that the analogous definitions with rank are equivalent, so we do not distinguish them.

 The terms \textit{generic} and \textit{independent} (applied to tuples) are always understood in the sense of $\mathcal R$; we distinguish the analogous notions in $\mathcal M$ (which we use less frequently) with prefixes (e.g. $\mathcal M$-generic). Similarly, we use $\operatorname{acl}$ to denote model-theoretic algebraic closure in $\mathcal R$, and $\acl_{\mathcal M}$ when referring to $\mathcal M$.

 \subsection{Families of Plane Curves}

 For us, a \textit{plane curve} always refers to a rank 1 $\mathcal M$-definable subset of $M^2$. We do not assume plane curves are strongly minimal. A plane curve $C\subset M^2$ is \textit{non-trivial} if both projections $C\rightarrow M$ are finite-to-one, if and only if $C$ does not almost contain $\{m\}\times M$ or $M\times\{m\}$ for any $m\in M$.
 
 If $C$ is a strongly minimal plane curve, we let $\operatorname{Cb}(C)$ denote the canonical base of the generic type of $C$, computed in the structure $\mathcal M$; recall this means a tuple $c\in\mathcal M^{\textrm{eq}}$ (unique up to interdefinability) such that any automorphism $\sigma$ of $\mathcal M$ fixes $c$ if and only if $\sigma(C)$ is almost equal to $C$. Recall that $\operatorname{Cb}(C)\in\mathcal M^{\textrm{eq}}$ always exists (see \cite{PillayBook}, Chapter 2, Claim 2.2). Moreover, since $\mathcal M$ is not locally modular, we have (see \cite{PillayBook}, Proposition 3.2):
 
 \begin{fact}\label{F: nlm via cbs} For any small set $A$ and integer $k$ there is a strongly minimal plane curve $C$ such that $\rk(\operatorname{Cb}(C)/A)\geq k$.
 \end{fact}
 
 An $\mathcal M$-definable family $\mathcal C=\{C_t:t\in T\}$ of plane curves is called \textit{almost faithful} if for each $t\in T$ there are only finitely many $t'\in T$ such that $C_t\cap C_{t'}$ is infinite. Given such a family $\mathcal C$, the \textit{graph} of $\mathcal C$ is the set $C=\{(x,t)\in M^2\times T:x\in C_t\}$. Fact \ref{F: nlm via families} is a restatement of Fact \ref{F: nlm via cbs} in this language (see \cite[Lemma 2.42]{CasACF0}):
 
 \begin{fact}\label{F: nlm via families} For each $r\geq 0$ there is an $\mathcal M$-definable almost faithful family of plane curves indexed by a large subset of $M^r$. 
 \end{fact}

\begin{convention}\label{C: 0 def families} Absorbing parameters if necessary, we will assume for each $r$ that at least one family as in Fact \ref{F: nlm via families} is $\emptyset$-definable in $\mathcal M$ (though we will only need this for $r=1$).
\end{convention}

 The precise relationship between Facts \ref{F: nlm via cbs} and \ref{F: nlm via families} is the following (see \cite[Lemma 2.24]{CasACF0}):
 
 \begin{fact}\label{F: cb vs af} Suppose $\{C_t:t\in T\}$ is almost faithful and $\mathcal M$-definable over $A$, with $\rk(T)=r$. Then for generic $t\in T$ and $S$ any strongly minimal component of $C_t$, we have $\rk(\operatorname{Cb}(S)/A)=r$. In fact, $\cb(S)$ is $\mathcal M$-interalgebraic with $t$ over $A$.
 \end{fact}

 \subsection{Dualizable Families}
 
 In \cite{CasACF0}, `good' families were defined as those which have `duals.' Since it seems the topic is of use beyond \cite{CasACF0}, we will rename the term as `dualizable':
 
 \begin{definition} An $\mathcal M$-definable almost faithful family $\mathcal C=\{C_t:t\in T\}$ of plane curves is \textit{dualizable} if $T$ is a large subset of $M^2$, and each of the sets $\{t:x\in C_t\}$ (for $x\in M^2$) is either empty or of rank 1. 
 \end{definition}
 
 If $\mathcal C$ is a dualizable family, then we can form its \textit{dual} family by switching the variables: we let $\mathcal C^{\vee}=\{C_x^{\vee}:x\in T^{\vee}\}$, where $C_x^{\vee}=\{t:x\in C_t\}$ and $T^{\vee}$ is the set of $x$ for which $C^{\vee}_x$ is non-empty. The following are proved in \cite{CasACF0} (Lemmas 2.44 and 2.50):
 
 \begin{fact}\label{F: good families} Let $\mathcal C=\{C_t:t\in T\}$ be an $\mathcal M$-definable almost faithful family of plane curves, with graph $C$.
 \begin{enumerate}
     \item If $T$ is a large subset of $M^2$, then there is a dualizable family whose graph is a large subset of $C$. In fact, there are non-generic $\mathcal M$-definable sets (over no additional parameters) $Z_1,Z_2\subset M^2$ such that $C-(Z_1\times M^2)-(M^2\times Z_2)$ is the graph of a dualizable family of plane curves.
     \item If $\mathcal C$ is dualizable, then the dual $\mathcal C^{\vee}$ is also dualizable, and in particular almost faithful.
\end{enumerate}
\end{fact}

Finally, we will need a strengthening of dualizability, called \textit{excellence} in \cite{CasACF0} and \cite{CaHaYe}:

\begin{definition} Let $\mathcal C=\{C_t:t\in T\}$ be a dualizable family of plane curves with graph $C\subset M^4$. We say that $\mathcal C$ is \textit{excellent} if each of the curves in $\mathcal C$ and $\mathcal C^\vee$ is non-trivial -- or equivalently, each of the four projections $C\rightarrow M^3$ is finite-to-one.
\end{definition}

It is proven in \cite{CasACF0} (Proposition 2.52) that every non-locally modular strongly minimal structure admits an excellent family. Thus, we also stipulate:

\begin{convention}\label{C: 0 def excellent family} We assume throughout that there is an $\mathcal M(\emptyset)$-definable excellent family of plane curves in $\mathcal M$. 
\end{convention}

 \subsection{Weak Detection of Closures}
 The main step in the proof in \cite{CasACF0} is a technical result on the detection (in the language of $\mathcal M$) of certain closure points of $\mathcal M$-definable sets (\cite[Proposition 8.6]{CasACF0}). This result was subsequently treated axiomatically in \cite{CaHaYe}, under the name `weak detection of closures' (\cite[Theorem 4.10]{CaHaYe}). Since we assume $(\mathcal R,\tau)$ has enough open maps, we get this statement directly from \cite{CaHaYe}:
 
 \begin{fact}\label{F: closure} $\mathcal M$ weakly detects closures: Let $X\subset M^n$ be $\mathcal M$-definable over $A$, of rank $r$. Let $\hat x=(x_1,...,x_n)\in M^n\cap\overline X$. Assume that each $x_i$ is generic in $M$ over $\emptyset$. Then $\rk(\hat x/A)\leq r$. Moreover, if $\rk(\hat x/A)=r$, and $i\neq j$ are such that for any generic $a=(a_1,...,a_n)\in X$ over $A$ the coordinates $a_i,a_j$ are $\mathcal M$-independent over $A$, then the coordinates $x_i,x_j$ are $\mathcal M$-independent over $A$.
 \end{fact}

\section{Warm-up: The Case of Groups}\label{group section}

We now move toward the proof of Theorem \ref{T: intro}. Recall that we assume $(\mathcal R,\tau,\mathcal S)$ has manifold ramification purity of some fixed order $d_0$. In this language, our general goal is to show that $\dim(M)\leq d_0$. For technical reasons, we will not realize this goal in complete generality. The potential problem is that, to combine our section-by-section results into the main theorem, one needs an adequate description of interpretable groups in $\mathcal R$. As this is no problem in the o-minimal case, we will not worry about it, and will simply give our section-by-section results.

In this section, we prove $\dim(M)\leq d_0$ assuming $M$ carries a definable `loric topological' group structure with finite 2-torsion (Theorem \ref{T: group case}). This will serve as a good model for the ensuing more general argument. In fact, however, we will need Theorem \ref{T: group case} in the end, so this section is not merely included for presentation.

Later on, we will abstract a general strategy from the proof of Theorem \ref{T: group case}. However, the abstraction process is not immediately clear. In fact, the proof below can be obtained from the general strategy via some simplifications that are only available in the presence of a group. We will not dwell on this, and will simply present the simplified argument (as it is rather slick), saving the abstraction process for later sections.

\begin{assumption}\label{A: group ass} \textbf{Throughout Section \ref{group section}, we assume that there is an $\mathcal M(\emptyset)$-definable group structure on the set $M$, and that the group operation and inverse are given by loric maps (in particular, $M$ is loric).}
\end{assumption}

\begin{theorem}\label{T: group case} Under Assumption \ref{A: group ass}, we have $\dim(M)\leq d_0$.
\end{theorem}
\begin{proof} Since strongly minimal groups are abelian, we use additive notation for the group operation, and thus assume $\mathcal M$ is an expansion of the group $(M,+)$. Let 0 be the identity of $(M,+)$. Abusing notation, we will also write 0 for the identity in $M^2$. Before we start, let us note:

\begin{lemma}\label{L: group is manifold}
\begin{enumerate}
    \item The difference map $(x,y)\mapsto x-y$ is loric.
    \item For each $a\in M$, the translation map $x\mapsto x+a$ is loric.
    \item $M$ is a loric $\dim(M)$-manifold.
\end{enumerate}
\end{lemma}
\begin{proof} 
\begin{enumerate}
    \item We use Lemma \ref{L: strong basic} repeatedly. By the projection property, $(x,y)\mapsto x$ and $(x,y)\mapsto y$ are loric. By the composition property (and the fact that the inverse operation is loric), $(x,y)\mapsto -y$ is loric. By the coordinate-by-coordinate property, $(x,y)\mapsto(x,-y)$ is loric. Now compose with the group operation to get $(x,y)\mapsto x-y$.
    \item The graph of translation by $a$ is the preimage of $-a$ under the difference map. So use (a), the preimage property, and the fact that finite sets are loric (Lemma \ref{L: strong basic}).
    \item By Lemma \ref{L: manifold genericity}, every generic point of $M$ belongs to $M^{lman}$. In particular (looking at one such point), there is an open subset $U\subset M$ which is a loric $\dim(M)$-manifold. By (b) and loric homeomorphism invariance (Lemma \ref{L: manifold preservation}(2)), every translate $U+x$ for $x\in M$ is a loric $\dim(M)$-manifold. The set of all translates forms an open cover of $M$. So by locality (Lemma \ref{L: manifold locality}), $M$ is a loric $\dim(M)$-manifold. 
\end{enumerate}
\end{proof}

Let us now give an imprecise summary of the proof. First, we construct an (approximately) $\mathcal M$-definable loric map $f$ satisfying the hypotheses of Definition \ref{D: ramification purity} and having at least one ramification point. Roughly speaking, $f$ will be the sum map $+:C^2\rightarrow M^2$ on an appropriately chosen plane curve $C$ containing 0. The witness to ramification will be the point $0+0=0$, and the ramification will follow from the equation $x+0=0+x=x$ for $x$ near 0. 

Once we have the map $f$, we will exploit the tension coming from two bounds on the value $\dim(\operatorname{Ram}(f))$: a lower bound is given by ramification purity (Definition \ref{D: ramification purity}), and an upper bound will follow from weak detection of closures (Fact \ref{F: closure}). Combining these two bounds yields an inequality simplifying to $\dim(M)\leq d_0$.

Because we work near the origin, in applying Fact \ref{F: closure} we cannot insist that each coordinate of the specified point be generic in $M$. Thus it will be convenient to know that, in the presence of the group structure, coordinate-wise genericity can be dropped:

\begin{lemma}\label{L: closure for groups} Let $X\subset M^n$ be $\mathcal M$-definable over $A$ of rank $r$, and let $\hat x=(\hat x_1,...,\hat x_n)\in M^n\cap\overline X$. Then $\rk(\hat x/A)\leq r$. Moreover, if $\rk(\hat x/A)=r$, and $i\neq j$ are such that for all generic $x=(x_1,...,x_n)\in X$ over $A$ the coordinates $x_i,x_j$ are $\mathcal M$-independent over $A$, then $\hat x_i$ and $\hat x_j$ are $\mathcal M$-independent over $A$.
\end{lemma}
\begin{proof} By translating by a generic, then applying Fact \ref{F: closure}, then translating back.

More precisely, let $g\in M^n$ be generic over $A$. Then $Y=X+g$ is $\mathcal M$-definable over $Ag$ of rank $r$, and $\hat y=\hat x+g\in\overline Y$. Moreover, it is easy to see by the genericity of $g$ that each $\hat y_i$ is generic in $M$ over $\emptyset$. Now apply Fact \ref{F: closure} to $\hat y$ and $Y$, with parameter set $Ag$. Since $\hat x$ and $\hat y$ are (coordinate by coordinate) $\mathcal M$-interdefinable over $Ag$, it follows easily that the conclusion of Fact \ref{F: closure} still holds of $\hat x$ and $X$ over $Ag$. But since $g$ is independent from $A\hat x$ by definition, it moreover follows easily that the $g$ can be dropped from the parameter set, thereby reducing to the original statement of the lemma.
\end{proof}

We now choose the appropriate plane curve $C$ described above.

\begin{definition} Let $C\subset M^2$ be a strongly minimal place curve.
\begin{enumerate}
    \item By the \textit{stabilizer} of $C$, denoted $\operatorname{Stab}(C)$, we mean the ($\mathcal M$-definable) subgroup of $M^2$ consisting of those $g$ such that $C+g$ is almost equal to $C$. 
    \item By $-C$, we mean the set $\{-x:x\in C\}$. Similarly, given $t\in M^2$, the sets $t+C=C+t$ and $t-C$ are the translates of $C$ and $-C$ (respectively) by $t$.
\end{enumerate}
\end{definition}

\begin{lemma}\label{L: nonaffine plane curve} There is a strongly minimal plane curve $C\subset M^2$ satisfying the following:
\begin{enumerate}
    \item $\operatorname{Stab}(C)$ is finite.
    \item $C\cap-C$ is finite.
    \item $0\in C^{lman}$.
\end{enumerate}
\end{lemma}
\begin{proof} It is automatic from non-local modularity that we can satisfy (1) \cite[Lemma 7.2 and Theorem 7.3(2)]{CasHas}. Now fix $C\subset M^2$ satisfying (1), and without loss of generality assume it is $\mathcal M(\emptyset)$-definable. Fix $a\in C$ generic, and let $C'=C-a$. Translation does not affect the stabilizer, so we still have (1). Moreover, since $a$ is generic, we have $a\in C^{lman}$ (Lemma \ref{L: manifold genericity}), and thus $a-a=0\in(C')^{lman}$ (Lemmas \ref{L: manifold preservation}(2) and \ref{L: group is manifold}(1)). So we get (3).

For (2), suppose $C'\cap-C'$ is infinite. By strong minimality, the two sets are almost equal. Thus $C-a$ is almost equal to $a-C$, or equivalently $-C$ is almost equal to $C-2a$. Now the set of $b$ with $-C$ almost equal to $C-b$ is a coset of $\operatorname{Stab}(C)$, so is finite. Moreover, this set is $\mathcal M(\emptyset)$-definable (because $C$ is). So $2a\in\acl_{\mathcal M}(\emptyset)$. Since $a\in M$ is generic, we get $a\notin\acl_{\mathcal M}(2a)$, and thus $x\mapsto 2x$ is not finite-to-one on $M$. By strong minimality, $M$ is a 2-torsion group, contradicting Assumption \ref{A: group ass}.
\end{proof}

Fix $C$ as in Lemma \ref{L: nonaffine plane curve}, and without loss of generality assume $C$ is $\mathcal M(\emptyset)$-definable. We next construct the function $f$ described in the summary. Naively, we want $f$ to be $+:C^2\rightarrow M^2$. In the end, we need to restrict $+$ to a neighborhood of $(0,0)$ satisfying the hypotheses of ramification purity. Thus, strictly speaking, we sacrifice the $\mathcal M$-definability of $f$; but the general strategy will still work.

First, we let $Z$ be the set of all $z\in M^2$ such that the equation $x+y=z$ has infinitely many solutions with $x,y\in C$. Rewriting $x+y=z$ as $x-z=-y$, one gets that $z\in Z$ if and only if $C-z\cap-C$ is infinite, which by strong minimality is equivalent to the almost equality of $C-z$ and $-C$. Now as in the proof of Lemma \ref{L: nonaffine plane curve}, the finiteness of $\operatorname{Stab}(C)$ implies that $Z$ is finite (namely it is either empty or a coset of $\operatorname{Stab}(C)$). Moreover, note that $0\notin Z$ by Lemma \ref{L: nonaffine plane curve}(2). 

Lemmas \ref{L: addition finite to one} and \ref{L: ramification at 0} set up our application of ramification purity:

\begin{lemma}\label{L: addition finite to one} Restricted to a small enough neighborhood of $(0,0)\in C^2$, the map $+:C^2\rightarrow M^2$ is a finite-to-one loric map of loric $(2\cdot\dim M)$-manifolds, whose image is disjoint from $Z$.
\end{lemma}
\begin{proof} Loricness is automatic by the restriction property (Lemma \ref{L: strong basic}(\ref{L: strong rest})). It is also automatic that the target $M^2$ is a loric $(2\cdot\dim(M))$-manifold, by Lemma \ref{L: group is manifold} and the preservation of loric manifolds under products (Lemma \ref{L: manifold preservation}(1)).

We now find an appropriate neighborhood of $(0,0)$ satisfying the other requirements. Since $Z$ is finite and $0\notin Z$, we can first delete the preimage of $Z$. This produces a finite-to-one map with image disjoint from $Z$. 

It remains only to find a small enough neighborhood of $(0,0)$ which is a loric $(2\cdot\dim(M))$-manifold. But $\dim(C^2)=2\cdot\dim(M)$ (by Fact \ref{F: dim rk}), so it suffices to show that $(0,0)\in(C^2)^{lman}$; and this follows from $0\in C^{lman}$ (Lemma \ref{L: nonaffine plane curve}(3)) and the preservation of loric manifolds under products (Lemma \ref{L: manifold preservation}(1)).
\end{proof}

\begin{notation} Let $f:X\rightarrow M^2$ denote the restriction of $+$ to a small neighborhood $X$ of $(0,0)$ in $C^2$ satisfying the conclusions of Lemma \ref{L: addition finite to one}.
\end{notation}

\begin{lemma}\label{L: ramification at 0} $(0,0)$ is a ramification point of $f$.
\end{lemma}

\begin{proof} 
Fix any neighborhood $V$ of $(0,0)$ in $X$. We show that $+:V\rightarrow M^2$ is not injective.

We may assume $V=U^2$ for some neighborhood $U$ of $0$ in $C$. Since $0\in C^{lman}$, we may assume $U$ is a loric $\dim(C)$-manifold (Lemma \ref{L: manifold preservation}(3)). In particular, $U$ is definable and $\dim(U)=\dim(C)>0$.

Now since $\dim(U)>0$, there is some element $0\neq u\in U$. But then $(0,u)\neq(u,0)\in U^2$, and $0+u=u+0=u$, witnessing ramification.
\end{proof}

We now apply purity of ramification to obtain our `lower bound' on $\dim(\operatorname{Ram}(f))$:

\begin{lemma}\label{L: lots of ramification in group} $\dim(\operatorname{Ram}(f)\geq 2\cdot\dim(M)-d_0$.
\end{lemma}
\begin{proof} By the choice of $X$, $\dim(X)=2\cdot\dim(M)$ (see Lemma \ref{L: addition finite to one}). So by ramification purity (Assumption \ref{A: R and M}(1) and Definition \ref{D: ramification purity}), $\operatorname{Ram}(f)$ is either empty or of dimension at least $2\cdot\dim(M)-d_0$. By Lemma \ref{L: ramification at 0}, $\operatorname{Ram}(f)$ is non-empty, so $\dim(\operatorname{Ram}(f))\geq 2\cdot\dim(M)-d_0$.
\end{proof}

We now move toward an upper bound on $\dim(\operatorname{Ram}(f))$ using weak detection of closures. We will work with the following set:

\begin{notation} Let $D$ be the set of $(x,y,x',y')\in C^4$ such that:
\begin{enumerate}
    \item $(x,y)\neq(x',y')$
    \item $x+y=x'+y'\notin Z$.
    \item $x-x'\notin\operatorname{Stab}(C)$.
\end{enumerate}
\end{notation}

$D$ is $\mathcal M(\emptyset)$-definable and carries a natural sum map, given by $(x,y,x',y')\mapsto x+y=x'+y'$. Since $D$ is defined only away from $Z$ (i.e. item (2) of the definition), this sum map is finite-to-one. Thus $\rk(D)\leq\rk(M^2)=2$. Now the main point is:

\begin{lemma}\label{L: ram pt in closure groups} Let $(x,y)\in C^2$. If $(x,y)\in\operatorname{Ram}(f)$, then $(x,y,x,y)\in\overline D$.
\end{lemma}
\begin{proof} Let $U\times V\times U\times V$ be a neighborhood of $(x,y,x,y)$ in $C^4$. We find an element of $(U\times V\times U\times V)\cap D$.

By assumption, $(x,y)$ belongs to $X$, which is open in $C^2$; so by shrinking $U$ and $V$, we may assume $U\times V\subset X$. In particular, if $u\in U,v\in V$ then $u+v\notin Z$ (by the choice of $X$, see Lemma \ref{L: addition finite to one}).

Since $\operatorname{Stab}(C)$ is also finite, we may further assume that the difference map on $U^2$ does not attain any non-zero elements of $\operatorname{Stab}(C)$ (that is, that if $u\neq u'\in U$ then $u-u'\notin\operatorname{Stab}(C)$).

Finally, since $f$ ramifies at $(x,y)$, there are $(u,v)\neq(u',v')\in U\times V$ with $u+v=u'+v'$. Then the above two paragraphs show that $(u,v,u',v')\in D$, as desired. 
\end{proof}

We now obtain our upper bound on $\dim(\operatorname{Ram}(f))$:

\begin{lemma}\label{L: noninjective point is dependent groups} $\dim(\operatorname{Ram}(f))\leq\dim(M)$.
\end{lemma}
\begin{proof}
    Suppose $f$ and $X$ are $\mathcal R(A)$-definable, and let $(\hat x,\hat y)\in\operatorname{Ram}(f)$ be generic over $A$. It suffices to show that $\dim(\hat x\hat y/A)\leq\dim(M)$. We will in fact show $\rk(\hat x\hat y/\emptyset)\leq 1$. Then by Fact \ref{F: dim rk}, $$\dim(\hat x\hat y/A)\leq\dim(\hat x\hat y/\emptyset)\leq\dim(M)\cdot 1=\dim(M),$$ so this is certainly enough.
    
    To show $\rk(\hat x\hat y)\leq 1$, we work in cases according to the value $\rk(D)$. From above, $\rk(D)\leq 2$. First suppose $\rk(D)\leq 1$. By Lemma \ref{L: ram pt in closure groups}, $(\hat x,\hat y,\hat x,\hat y)\in\overline D$. Then by Lemma \ref{L: closure for groups}, $\rk(\hat x\hat y)\leq\rk(D)\leq 1$, as desired.

    Now suppose $\rk(D)=2$. By similar reasoning (or the fact that $\hat x,\hat y\in C$), we have $\rk(\hat x\hat y)\leq 2$. So, arguing toward a contradiction, assume $\rk(\hat x\hat y)=2=\rk(D)$. We will apply the final clause of weak detection of closures (Lemma \ref{L: closure for groups}). First we check:

    \begin{claim}
        For generic $(x,y,x',y')=(x_1,x_2,y_1,y_2,x_1',x_2',y_1',y_2')\in D$, the coordinates $x_1$ and $x_1'$ are $\mathcal M$-independent.
    \end{claim}
    \begin{proof}
        We want to show that $\rk(x_1x_1')=\rk(xyx'y')=2$. We do this by showing that $(x,y,x',y')\in\acl_{\mathcal M}(x_1x_1')$. Let $w=x+y=x'+y'\in M^2-Z$. 

        Since $C$ has finite stabilizer, it is non-trivial (i.e. each projection $C\rightarrow M$ is finite-to-one). Indeed, otherwise $C$ would be almost equal to a horizontal or vertical line, which have infinite stabilizers. Thus $x_2\in\acl_{\mathcal M}(x_1)$ and $x_2'\in\acl_{\mathcal M}(x_1')$. In particular, $(x,x')\in\acl_{\mathcal M}(x_1x_1')$. Next, by definition we have $w\in(x+C)\cap(x'+C)$; and also by definition we have $x-x'\notin\operatorname{Stab}(C)$; thus $(x+C)\cap(x'+C)$ is finite, and so $w\in\acl_{\mathcal M}(xx')$. Finally, since $w\notin Z$ and $x+y=x'+y'=w$, we have $(x,y,x',y')\in\acl_{\mathcal M}(w)$. So, putting everything together, we obtain $(x,y,x',y')\in\acl_{\mathcal M}(x_1x_1')$, as desired.
    \end{proof}

    Now by the claim and Lemma \ref{L: closure for groups}, the corresponding coordinates of $(\hat x,\hat y,\hat x,\hat y)$ are $\mathcal M$-independent -- namely $\hat x_1$ and $\hat x_1$. But equal elements are only independent if they are algebraic -- so $\hat x_1\in\acl_{\mathcal M}(\emptyset)$.
    
    Now since $\hat x,\hat y\in C$, and by non-triviality, we have $(\hat x,\hat y)\in\acl_{\mathcal M}(\hat x_1\hat y_1)$. Since $\hat x_1\in\acl_{\mathcal M}(\emptyset)$, this means $$\rk(\hat x\hat y)\leq\rk(\hat x_1)+\rk(\hat y)\leq 0+1,$$ again implying the lemma.
\end{proof}
    Finally, we combine our two bounds to prove Theorem \ref{T: group case}. Namely, by Lemmas \ref{L: lots of ramification in group} and \ref{L: noninjective point is dependent groups}, we have $$2\cdot\dim(M)-d_0\leq\dim(\operatorname{Ram}(f))\leq\dim(M).$$ Thus $2\cdot\dim(M)-d_0\leq\dim(M)$, and rearranging gets $\dim(M)\leq d_0$.
\end{proof}

\section{The General Strategy}\label{S: informal}

We now revisit the proof of Theorem \ref{T: group case} and abstract a general strategy for Theorem \ref{T: intro} that does not use a group operation. This section will be an informal discussion. The actual proof begins in Section \ref{S: general case}.

\subsection{From Addition to the Family of Translates}

Let us temporarily resume the setup in the proof of Theorem \ref{T: group case}: we have a strongly minimal plane curve $C$, and we consider the sum map on $C^2$, or more precisely the fibers of this map -- the sets $\{(x,y)\in C^2:x+y=z\}$ for $z\in M^2$. In the general case, our only tool is a dualizable family of plane curves. Thus, we want to rewrite the proof in terms of such a family. The family we use is $\{C+t:t\in M^2\}$, the translates of $C$. (This is almost faithful since $\operatorname{Stab}(C)$ is finite).

As in the proof of Theorem \ref{T: group case}, the fiber above $z\in M^2$ in $+:C^2\rightarrow M^2$ is essentially $C\cap(z-C)$. If $x\in C\cap(z-C)$, there is some (obviously unique) $y\in C$ with $x=z-y$, and thus $x+y=z$; and if $x+y=z$ with $x,y\in C$, then by definition $x=z-y\in C\cap(z-C)$. 

Thus, we view the fibers of $+:C^2\rightarrow M^2$ as intersections of $C$ with translates of $-C$. To add some symmetry, it is harmless to tanslate $C$ as well. That is, let $\mathcal D=\{D_t:t\in M^2\}$ and $\mathcal E=\{E_t:t\in M^2\}$ where $D_t=C+t$ and $E_t=t-C$. Then the fibers of $+:C^2\rightarrow M^2$ are captured by the family of all intersections $D_s\cap E_t$.

Our goal is to describe the objects in the proof of Theorem \ref{T: group case} using only $\mathcal D$. At this point, it is enough to express $\mathcal E$ in terms of $\mathcal D$. The punch line: $\mathcal D$ and $\mathcal E$ are dual. Thus $\mathcal E=\mathcal D^\vee$, and we can redo the proof in terms of the intersections $D_s\cap D^{\vee}_t$. To see this, let $x,y\in M^2$ be arbitrary. Then: $$x\in D_y\iff x\in C+y\iff x-y\in C\iff y-x\in -C$$ $$\iff y\in x-C\iff y\in E_x.$$ 

\subsection{Extracting the Appropriate Map}

We thus wish to study pairwise intersections of curves in $\mathcal D$ and $\mathcal D^{\vee}$. After a harmless change of variables, let us express this as follows. Let $$I=\{(x,y,z)\in M^6:y\in C+x\wedge z\in C+y\}.$$ Note that $(x,y,z)\in I$ is equivalent to $y\in D_x\cap D^\vee_z$. So the family of intersections of $\mathcal D$ and $\mathcal D^\vee$ is captured by the projection $$\pi:I\rightarrow M^4,\;\;\;\;(x,y,z)\mapsto(x,z).$$ The original ramification point $0+0=0$ translates in this language to $(0,0,0)\in I$ and $\pi(0,0,0)=(0,0)$. The witnessing equations $0+u=u+0$ say that for $u\in C$ close to 0 we have $$0,u\in C\cap(u-C),$$ so that $$(0,0,u),(0,u,u)\in I$$ and $$\pi(0,0,u)=\pi(0,u,u)=(0,u).$$ 

In the general case, we want to use tuples that are as generic as possible (rather than honing in on the identity, which is natural in a group). To help with this, consider the image of the above data under translation by a generic of $M^2$. First, let $x\in M^2$ be generic. Since $0\in C$, we have $x\in(C+x)\cap(x-C)$, so that $(x,x,x)\in I$ and $\pi(x,x,x)=(x,x)$. 

In fact, $(x,x,x)$ is still a ramification point of $\pi$. To see this, let $u\in C$ be close to 0. Since $0,u\in C$, one gets $$x,x+u\in(C+x)\cap((x+u)-C).$$ Now let $z=x+u$, so $z$ is close to $x$. Then the above translates to $$(x,x,z),(x,z,z)\in I$$ and $$\pi(x,x,z)=\pi(x,z,z)=(x,z),$$ witnessing ramification.

\subsection{The Abstract Setup}

Now suppose $\mathcal M$ is not necessarily a group. The following has been motivated:

Let $\mathcal D=\{D_t:t\in T\}$ be a dualizable family of plane curves. Analogously define $$I=\{(x,y,z)\in M^6:y\in D_x\wedge z\in D_y\}$$ and $$\pi:I\rightarrow M^4,\;\;\;\;(x,y,z)\mapsto(x,z).$$ Then for generic $x\in M^2$, we hope to find a ramification point of $\pi$ of the form $\pi(x,x,x)=(x,x)$, witnessed by points of the form $$\pi(x,x,z)=\pi(x,z,z)=(x,z)$$ for $z$ near $x$.

To guarantee $(x,x,x)\in I$, we need to choose $\mathcal D$ so that $x\in D_x$ for generic $x$. To be able to apply purity of ramification (Definition \ref{D: ramification purity}), we even need that $(x,x,x)\in I^{lman}$ and $\pi$ is finite-to-one near $(x,x,x)$. The main portion of the proof of Theorem \ref{T: intro} amounts to finding $\mathcal D$ with these properties.

\subsection{What We Need for the Construction to Work}

 In the end, the condition $(x,x,x)\in I^{lman}$ seems too strong. This has to do with the failure of definable (finite) choice in $\mathcal M$: the natural way to satisfy $x\in D_x$ is to define each $D_x$ by `choosing' a curve through $x$ from a bigger family. Without definable choice, one has to choose several such curves and take their union -- and this introduces branch points, for both $\mathcal D$ and $I$.
 
 Instead, we require that the identity $x\in D_x$ be witnessed by a uniformly $\mathcal R$-definable collection of `loric manifold branches' of the family $\mathcal D$. We then use such a family to construct a loric manifold branch of $I$ near $(x,x,x)$, and apply purity of ramification to this branch. This is similar to the group case above, where we sacrificed $\mathcal M$-definability for access to ramification purity.

 Arranging that $\pi$ is locally finite-to-one is more complicated. In the group case, we did this by exploiting the uniformity of the group structure (essentially, the infinite fibers of $\pi$ were all controlled by the finite set $Z$). Without a group, it is difficult to `map out' all infinite fibers of $\pi$. 
 
 In the end, our approach is to make \textit{all} fibers finite, so there is nothing to map out. This makes the construction of $\mathcal D$ much harder, but the eventual proof easier. One can think of this requirement as saying that $\mathcal D$ and $\mathcal D^\vee$ are genuinely different families of curves -- i.e. they have no members in common, so they are not just permutations of each other.

 \subsection{The Strong Identity Property and Self-Duality}

The following are now motivated.

 \begin{definition}\label{D: identity property} Let $\mathcal D=\{D_t:t\in T\}$ be a dualizable family of plane curves. We say that $\mathcal D$ has the \textit{strong identity property} if for generic $(a,b)\in M^2$, there are $\mathcal R$-definable relative neighborhoods $A\subset M$ of $a$ and $B\subset M$ of $b$ in $M$, and a loric map $h:A\times B\times A\rightarrow M$, such that:
		\begin{enumerate}
			\item $A$ and $B$ are loric $\dim(M)$-manifolds.
			\item If $(a',b',a'')\in A\times B\times A$ and $b''=h(a',b',a'')$ then $(a'',b'')\in D_{(a',b')}$.
			\item If $(a',b')\in A\times B$ then $h(a',b',a')=b'$.
		\end{enumerate}
 \end{definition}

 \begin{definition}\label{D: self dual} Let $\mathcal D=\{D_t:t\in T\}$ be a dualizable family of plane curves. We say that $\mathcal D$ is \textit{self-dual} if for generic $x\in M^2$ there is some $z\in M^2$ such that $D_x\cap D^\vee_z$ is infinite.
\end{definition}

\begin{remark} Since dualizable families (and their duals) are almost faithful, it is easy to see that the $z$ witnessing self duality at $x$ is $\mathcal M$-interalgebraic with $x$, so also generic. Thus $\mathcal D$ is not self-dual if and only if there are no $x$ and $z$, both generic in $M^2$, so that $D_x\cap D^\vee_z$ is infinite.
\end{remark}

Our informal discussion suggests that we can prove the following; the formal proof will occupy the next section:

\begin{theorem}\label{T: general case from family} Assume there is a dualizable family of plane curves in $\mathcal M$ which has the strong identity property and is not self-dual. Then $\dim M\leq d_0$.
\end{theorem}

In Section \ref{S: general case} we prove Theorem \ref{T: general case from family}. In Section \ref{construction section} we give a construction that either produces a group configuration, or a dualizable family that has the strong identity property and is not self-dual. Thus, we conclude Theorem \ref{T: intro} by either Theorem \ref{T: group case} or Theorem \ref{T: general case from family}.

\section{Proof of the General Case From an Appropriate Family}\label{S: general case}

We now prove Theorem \ref{T: general case from family} using the method described in the previous section.

\begin{proof}[Proof of Theorem \ref{T: general case from family}] Let $\mathcal D=\{D_t:t\in T\}$ be a dualizable family of plane curves in $\mathcal M$ which has the strong identity property and is not self-dual. Without loss of generality, assume $\mathcal D$ is $\mathcal M(\emptyset)$-definable.

\begin{notation} As discussed above, we fix the following throughout the proof.
\begin{itemize}
    \item Let $I$ be the set of $(x,y,z)\in M^6$ such that $y\in D_x$ and $z\in D_y$.
    \item Let $\pi:I\rightarrow M^4$ be the projection given by $(x,y,z)\mapsto(x,z)$.
    \item Let $Z$ be the set of $(x,z)$ with infinite fiber under $\pi$. 
    \item Let $\hat x=(\hat a,\hat b)$ be a fixed generic in $M^2$. 
\end{itemize}
\end{notation}

Automatically, we have:
\begin{itemize}
    \item $I$, $\pi$, and $Z$ are $\mathcal M(\emptyset)$-definable.
    \item The fibers of $\pi$ are the intersections $D_x\cap D^{\vee}_z$.
    \item $(\hat x,\hat x,\hat x)\in I$ (by the strong identity property).
\end{itemize}

We give some more basic properties:

\begin{lemma}\label{L: rk I} $\rk(I)=4$.
\end{lemma}
\begin{proof}
    For $y\in M^2$, the set of $(x,z)$ with $(x,y,z)\in I$ is just $D^\vee_y\times D_y$. By dualizability, the sets $D_w$ and $D^{\vee}_w$ (for $w\in M^2$) are all of rank $\leq 1$, and are generically of rank 1. Thus the fibers $D^\vee_y\times D_y$ are all of rank $\leq 2$, and are generically of rank 2. So $\rk(I)=\rk(M^2)+2=4$. 
\end{proof}

\begin{lemma}\label{L: Z small} $\rk(Z)\leq 1$.
    \end{lemma} 
    \begin{proof} If not, there is $(x,z)\in M^4$ with $D_x\cap D^\vee_z$ infinite and $\rk(xz)\geq 2$. Since $\mathcal D$ and $\mathcal D^{\vee}$ are almost faithful, $x$ and $z$ are $\mathcal M$-interalgebraic. So $\rk(x)=\rk(z)\geq 2$, implying that $x$ and $z$ are both $\mathcal M$-generic in $M^2$. This shows that $\mathcal D$ is self-dual, a contradiction.
    \end{proof}

\begin{lemma}\label{L: projection finite to one} $\pi$ is finite-to-one in a neighborhood of $(\hat x,\hat x,\hat x)$.
\end{lemma}
\begin{proof}
   If not, then $(\hat x,\hat x)\in\overline Z$. But each coordinate of $(\hat x,\hat x)$ is generic in $\mathcal M$ by construction, so Fact \ref{F: closure} applies. Combined with Lemma \ref{L: Z small}, we get $\rk(\hat x)\leq\rk(Z)\leq 1$, contradicting the choice of $\hat x$.
\end{proof}

 We now use the strong identity property to extract a loric map of loric manifolds (the analog of `$f$' from the group case).
 
	\begin{lemma}\label{L: submanifold of I} There is a loric $(4\cdot\dim(M))$-manifold $X\subset I$ such that $(\hat x,\hat x,\hat x)$ is a ramification point of the restriction of $\pi$ to $X$.
	\end{lemma}
	\begin{proof} Let $A,B,h$ be provided from Definition \ref{D: identity property}, for the family $\mathcal D$ and the point $\hat x=(\hat a,\hat b)$. By the continuity of $h$, there are neighborhoods $A'$ and $B'$ of $\hat a$ and $\hat b$ in $M$ so that $h(A'\times B'\times A')\subset B$. Without loss of generality assume $A'$ and $B'$ are loric $\dim(M)$-manifolds (Lemma \ref{L: manifold preservation}(3)). By the restriction property (Lemma \ref{L: strong basic}(\ref{L: strong rest})), $h$ is loric on $A'\times B'\times A'$.
 
    We define a loric embedding $$g:A'\times B'\times A'\times A\rightarrow M^6$$ which takes images in $I$, and let $X$ be the image of $g$. Given $$(a,b,a',a'')\in A'\times B'\times A'\times A,$$ let $b'=h(a,b,a')\in B$, and let $b''=h(a',b',a'')$. Then set $$g(a,b,a',a'')=(a,b,a',b',a'',b'').$$ 
    
    We now check that $X:=\operatorname{Im}(g)\subset M^6$ satisfies the requirements of the lemma. 
    
    \begin{claim} $X$ is a loric $(4\cdot\dim(M))$-manifold and $X\subset I$.   \end{claim} 
    \begin{proof} First, $g$ is loric. Indeed, $g$ is clearly continuous, and its graph is generated by the graph of $h$ by a sequence of products and intersections (so the graph of $g$ is also loric).
    
    Now $g$ has a continuous inverse $X\rightarrow A'\times B'\times A'\times A$, given by a coordinate projection. In particular, $g$ is a homeomorphism, and thus a loric homeomorphism. So $X$ is lorically homeomorphic to $A'\times B'\times A'\times A$, and thus $X$ is a loric $(4\cdot\dim(M))$-manifold (by the preservation of loric manifolds under products and loric homeomorphisms (Lemma \ref{L: manifold preservation}(1)-(2))).
    
    Now let $(a,b,a',b',a'',b'')=g(a,b,a',a'')\in X$. So $b'=h(a,b,a')$ and $b''=h(a',b',a'')$. Then by the choice of $h$ (Definition \ref{D: identity property}), this gives $(a',b')\in D_{(a,b)}$ and $(a'',b'')\in D_{(a',b')}$. In other words, $(a,b,a',b',a'',b'')\in I$.
    \end{proof}

    \begin{claim} $(\hat x,\hat x,\hat x)$ is a ramification point of the restriction of $\pi$ to $X$.
    \end{claim}
    \begin{proof} By the choice of $g$, we have $g(\hat a,\hat b,\hat a,\hat a)=(\hat a,\hat b,\hat a,\hat b,\hat a,\hat b)$, and thus $(\hat x,\hat x,\hat x)\in X$. Now let $V$ be any neighborhood of $(\hat x,\hat x,\hat x)$ in $X$. Since $g$ is a homeomorphism, we have $V=g(U)$ for some neighborhood $U$ of $(\hat a,\hat b,\hat a,\hat a)$ in $A'\times B'\times A'\times A$. We many assume $U=Y\times Z\times Y\times Y$ where $Y$ and $Z$ are loric $\dim(M)$-manifolds. So $\dim(Y)=\dim(M)>0$, and thus there is $a\in Y$ with $a\neq\hat a$. Let $b=h(\hat a,\hat b,a)$, and let $z=(a,b)$. Then:
    
    \begin{enumerate}
        \item $(\hat a,\hat b,\hat a,a)\in Y\times Z\times Y\times Y$.
        \item $g(\hat a,\hat b,\hat a,a)=(\hat x,\hat x,z)$.
        \item $(\hat a,\hat b,a,a)\in Y\times Z\times Y\times Y$.
        \item $g(\hat a,\hat b,a,a)=(\hat x,z,z)$.
        \item $(\hat x,\hat x,z)\neq\hat(\hat x,z,z)$ (since $\hat a\neq a$).
    \end{enumerate}

    By (1) and (2), $(\hat x,\hat x,z)\in V$. By (3) and (4), $(\hat x,z,z)\in V$. But $\pi(\hat x,\hat x,z)=\pi(\hat x,z,z)=(\hat x,z)$, and by (5), this witnesses ramification.
 \end{proof}
 We have now proved Lemma \ref{L: submanifold of I}.
 \end{proof}

 Note that $(\hat x,\hat x)\in(M^4)^{lman}$ (since $\hat x\in M^2$ is generic and loric manifolds are closed under products). Thus there is a loric $(4\cdot\dim(M))$-manifold $Y\subset M^4$, open in $M^4$, with $(\hat x,\hat x)\in Y$. By the locality of loric manifolds, we may assume $f(X)\subset Y$.

 \begin{notation}
     \textbf{For the rest of the proof of Theorem \ref{T: general case from family}, let $X$ be as in Lemma \ref{L: submanifold of I}, and let $Y\subset M^4$ be a $(4\cdot\dim(M))$-manifold which is open in $M^4$ and satisfies $f(X)\subset Y$.}
 \end{notation}

     Since $X$ and $Y$ are loric (and $f$ is just the projection), it follows that $f$ is loric (Lemma \ref{L: strong basic}(\ref{L: strong projection property})). As in the proof of Theorem \ref{T: group case}, we now develop upper and lower bounds for the ramification locus of $f$. Without a group, Fact \ref{F: closure} only works on coordinate-wise generic points, so we cannot explicitly bound $\dim(\operatorname{Ram}(f))$. Instead, we have to work at the level of types. First, we apply purity of ramification to get a lower bound:
    
    \begin{lemma}\label{L: lots of ramification} There is an element $(u,v,w)\in I$ such that the following hold:
    \begin{enumerate}
        \item Each of $u$, $v$, and $w$ realizes $\operatorname{tp}_{\mathcal R}(\hat x/\emptyset)$, and is thus generic in $M^2$.
        \item $(u,v,w)\in\operatorname{Ram}(f)$.
        \item $\dim(uvw)\geq 4\cdot\dim(M)-d_0$.
    \end{enumerate}
    \end{lemma}
    \begin{proof} Let $S\subset M^2$ be any $\mathcal R(\emptyset)$-definable set containing $\hat x$. By compactness, it suffices to find $(u,v,w)\in S^3$ satisfying (2) and (3).
    
    By the frontier inequality, $S$ contains a neighborhood $U$ of $\hat x$ in $M^2$: otherwise, $\hat x\in\operatorname{Fr}(M^2-S)$, so $\dim(\hat x)<\dim(M^2-S)\leq\dim(M^2)$, contradicting that $\hat x$ is generic in $M^2$.
    
    Let $U$ be as above. We may assume $U$ is definable and open in $M^2$. Thus $X\cap U^3$ is a loric $(4\cdot\dim(M))$-manifold (Lemma \ref{L: manifold preservation}(3)). Let $g$ be the restriction of $\pi$ to $X\cap U^3$. By the restriction property, $g$ is loric (Lemma \ref{L: strong basic}(\ref{L: strong rest})). By Lemma \ref{L: projection finite to one}, $g$ satisfies the requirements of ramification purity (Definition \ref{D: ramification purity}), so either $\operatorname{Ram}(g)=\emptyset$ or $\dim(\operatorname{Ram}(g))\geq 4\cdot\dim(M)-d_0$. By Lemma \ref{L: submanifold of I}, $(\hat x,\hat x,\hat x)\in\operatorname{Ram}(g)$, so $\dim(\operatorname{Ram}(g))\geq 4\cdot\dim(M)-d_0$. Now let $(u,v,w)$ be generic in $\operatorname{Ram}(g)$ over any set of parameters defining it. Then $(u,v,w)\in S^3$ and $(u,v,w)$ satisfies (3). Moreover, since $g$ is a restriction of $f$, clearly $\operatorname{Ram}(g)\subset\operatorname{Ram}(f)$, and thus $(u,v,w)$ satisfies (2). 
    \end{proof}

    \textbf{For the rest of the proof of Theorem \ref{T: general case from family}, fix $(u,v,w)=(u_1,u_2,v_1,v_2,w_1,w_2)$ as in Lemma \ref{L: lots of ramification}.} We now use Fact \ref{F: closure} to prove an `upper bound': $\dim(uvw)\leq 3\cdot\dim(M)$. Without a group, it will take some work to set up Fact \ref{F: closure}. The argument is essentially identical to the proof of Theorem 9.1 of \cite{CasACF0}.
    
    \begin{notation} For the rest of the proof, we define $P$ and $Q$ as follows. Let $Q$ be the set of $(x,y,y',z)\in M^8$ such that $y\neq y'$ and $(x,y,z),(x,y',z)\in I$. Then let $P$ be the set of $(x,y,y',z)\in Q$ such that $(x,z)\notin Z$.
    \end{notation}

    Note that given a point $q=(x,y,y',z)\in Q$, we will tacitly assume to label the eight coordinates of $q$ as $(x_1,x_2,y_1,y_2,y_1',y_2',z_1,z_2)$, without writing the full list each time.

    \begin{lemma}\label{L: rk P} $\rk(P)\leq 4$.
    \end{lemma}
    \begin{proof} Since the definition of $P$ insists that $(x,z)\notin Z$, the projection $P\rightarrow M^4$, $(x,y,y',z)\mapsto(x,z)$, is automatically finite-to-one. Thus $\rk(P)\leq\rk(M^4)=4$.
    \end{proof}

     \begin{lemma}\label{L: frontier of P} $(u,v,v,w)\in\operatorname{Fr}(P)$.
    \end{lemma}
    \begin{proof}
        It follows from $(u,v,w)\in\operatorname{Ram}(f)$ that $(u,v,v,w)\in\operatorname{Fr}(Q)$. So it suffices to show that $(u,w)\notin\overline Z$. But each $M$-coordinate of $(u,w)$ is generic in $M$ (Lemma \ref{L: lots of ramification}(1)). So if $(u,w)\in\overline Z$, then Fact \ref{F: closure} and Lemma \ref{L: Z small} give $\rk(uw)\leq\rk(Z)\leq 1$. This contradicts that $u$ is generic in $M^2$ (Lemma \ref{L: lots of ramification}(1)).
        \end{proof}   

        We would like to apply Fact \ref{F: closure} to $(u,v,v,w)\in\operatorname{Fr}(P)$, but we need to `preen' $P$ first so that Fact \ref{F: closure} gives us what we want. This is the point of Lemma \ref{L: avoiding E} below.

        First note that for $(x,y,y',z)\in P$, either $y_1\neq y_1'$ or $y_2\neq y_2'$. Let $P_1$ be the set of $(w,y,y',z)\in P$ with $y_1\neq y_1'$, and $P_2$ the set of $(x,y,y',z)\in P$ with $y_2\neq y_2'$. So $P=P_1\cup P_2$, and thus $(u,v,v,w)\in\operatorname{Fr}(P_i)$ for some $i=1,2$. Without loss of generality, assume $(u,v,v,w)\in\operatorname{Fr}(P_1)$.
    
    \begin{definition} Let $y_1\neq y_1'\in M$. We say that $y_1$ and $y_1'$ are \textit{extendable} if the set of extensions of $y_1,y_1'$ to an element of $P_1$ (with $y_1,y_1'$ in the appropriate coordinates) has rank at least 3.
    \end{definition}
    
    Let $E$ be the set of extendable pairs. Then $E$ is $\mathcal M(\emptyset)$-definable, and Lemma \ref{L: rk P} (and the definition of $E$) give $\rk(E)\leq 1$. Let $P'$ be the set of $(x,y,y',z)\in P_1$ with $(y_1,y'_1)$ not extendable.
    
    \begin{lemma}\label{L: avoiding E} $(u,v,v,w)\in\operatorname{Fr}(P')$.
    \end{lemma}
    \begin{proof}
        Since $(u,v,v,w)\in\operatorname{Fr}(P_1)$, it suffices to show that $(v,v)\notin\overline E$. Suppose toward a contradiction that $(v,v)\in\overline E$. Then $(v_1,v_1)\in\operatorname{Fr}(E)$ (since $E$ is disjoint from the diagonal). By the frontier inequality (using $\rk(E)\leq 1$ and Fact \ref{F: dim rk}), we get $\dim(v_1)<\dim E\leq\dim M$. This contradicts that $v_1$ is generic in $M$.
    \end{proof}
    
  We now obtain our upper bound:
    
    \begin{lemma}\label{L: rk uvw}
        $\rk(uvw)\leq 3$.
    \end{lemma}
    \begin{proof}
        Suppose not. Since $P'\subset P$, Lemma \ref{L: rk P} gives $\rk(P')\leq 4$. Moreover, by Lemma \ref{L: lots of ramification} all $M$-coordinates of $(u,v,v,w)$ are generic in $M$. By Fact \ref{F: closure}, the only possibility is $\rk(uvw)=\rk(P')=4$. We then move toward the second clause of Fact \ref{F: closure}. 
        \begin{claim} Let $(x,y,y',z)\in P'$ be generic. Then $y_1$ and $y_1'$ are $\mathcal M$-independent.
        \end{claim}
        \begin{proof} By genericity, $\rk(xyy'z)=4$. By definition of $P'$, the coordinates $y_1$ and $y_1'$ are not extendable, so that $\rk(xyy'z/y_1y_1')\leq 2$. Thus $\rk(y_1y_1')\geq 2$, i.e. $y_1$ and $y_1'$ are $\mathcal M$-independent $\mathcal M$-generics in $M$.
        \end{proof}
        By the claim and Fact \ref{F: closure}, the corresponding coordinates of $(u,v,v,w)$ must be $\mathcal M$-independent -- namely $v_1$ and $v_1$. But equal and independent elements must be algebraic, so $v_1\in\acl_{\mathcal M}(\emptyset)$. This contradicts that $v$ is generic in $M^2$.
        \end{proof}
 As in the proof of Theorem \ref{T: group case}, we now combine our upper and lower bounds to finish the proof of Theorem \ref{T: general case from family}. Namely, combining Lemmas \ref{L: lots of ramification} and \ref{L: rk uvw} (and using Fact \ref{F: dim rk} to replace $\rk$ with $\dim$), we have $$4\cdot\dim(M)-d_0\leq\dim(uvw)\leq 3\cdot\dim(M).$$ Thus $$4\cdot\dim(M)-d_0\leq 3\cdot\dim(M),$$ and rearranging gives $\dim(M)\leq d_0$.
    \end{proof}

    \section{Construction of the Family}\label{construction section}

    We now give a construction of a dualizable family of plane curves that has the strong identity property and is not self dual. The idea is to cleverly reparametrize the composition of two carefully chosen rank 1 families. To be precise, this only works if the composite family is (generically) almost faithful. If it isn't, we use a standard argument to show that $\mathcal M$ interprets a strongly minimal group -- thus we can use Theorem \ref{T: group case} instead of Theorem \ref{T: general case from family}. 
    \begin{theorem}\label{T: family exists} Either $\mathcal M$ interprets a strongly minimal group, or there is a dualizable family of plane curves in $\mathcal M$ which has the strong identity property and is not self-dual.
    \end{theorem}

    \begin{proof} First, by Conventions \ref{C: 0 def families} and \ref{C: 0 def excellent family}, and Fact \ref{F: nlm via cbs}, we may fix the following:

    \begin{itemize}
        \item $\mathcal C=\{C_t:t\in T\}$, an $\mathcal M(\emptyset)$-definable almost faithful family of non-trivial plane curves, where $T$ is large in $M$.
        \item $\mathcal D=\{D_u:u\in U\}$, an $\mathcal M(\emptyset)$-definable excellent family of plane curves.
        \item $V\subset U$, a strongly minimal set with $\rk(w)\geq 7$, where $w=\operatorname{Cb}(V)$. Editing finitely many points, we assume $V$ is $\mathcal M(w)$-definable.
        \item $b_0\in M$, generic over $w$.
    \end{itemize}

    \begin{notation} Throughout, we let $C$, $D$, and $D'$ be the graphs of $\mathcal C$, $\mathcal D$, and $\mathcal D'=\{D_v:v\in U\}$, respectively.
    \end{notation}

    \begin{lemma}\label{L: generically good} We may assume that all projections $C\rightarrow M^2$, $D\rightarrow M^3$, $D'\rightarrow M^2$, and $D\rightarrow M\times V$ are finite-to-one with large image.
    \end{lemma}
    \begin{proof} For the four projections $D\rightarrow M^3$ this is the definition of the excellence of $\mathcal D$. For the two rightmost projections $C\rightarrow M^2$, as well as the two projections $D'\rightarrow M\times V$, this follows from the non-triviality of each curve in each family. For the leftmost projection $C\rightarrow M^2$, we delete all elements of $M^2$ that belong to infinitely many $C_t$ (by almost faithfulness, there are only finitely many of them).
    
    Finally, we claim that $D'\rightarrow M^2$ is already finite-to-one, and thus has large image by rank considerations. Indeed, suppose $p\in M^2$ and there are infinitely many $v\in V$ with $p\in D_v$. Since $\mathcal D$ is excellent and $V$ is strongly minimal, the dual fiber $D^\vee_p$ has rank 1 and almost contains $V$. Thus $V$ is (up to almost equality) a strongly minimal component of $D^\vee_p$. This implies $w\in\acl_{\mathcal M}(p)$, so $\rk(w)\leq\rk(p)\leq 2$, contradicting that $\rk(w)\geq 7$.
    \end{proof}

    Assume moving forward that all projections from Lemma \ref{L: generically good} are finite-to-one with large image.

    Recall that if $S_1$ and $S_2$ are non-trivial plane curves, their \textit{composition} $S_2\circ S_1$ is the set of $(x,z)$ such that for some $y$ we have $(x,y)\in S_1$ and $(y,z)\in S_2$. It is again a non-trivial plane curve. We now use compositions to define (at least generically) the rank 2 family we will use for the proof:

    \begin{definition} For $(a,c)\in M^2$, define $E'_{ac}$ to be the union of all compositions $D_v\circ C_t$ with $v\in V$, $t\in T$, $(a,b_0)\in C_t$, and $(b_0,c)\in D_v$. Let $\mathcal E'$ be the family of $E'_{ac}$'s, and $E'\subset M^4$ the graph of $\mathcal E'$.
    \end{definition}

    Note that $\mathcal E'$ is $\mathcal M(wb_0)$-definable; each $E'_{ac}$ is either empty or a non-trivial plane curve; and for generic $(a,c)\in M^2$, $E'_{ac}$ is indeed a plane curve. We now split into cases according to whether $\mathcal E'$ is almost faithful.

    The following is a restatement of the group configuration theorem (see \cite{bou}) in the strongly minimal case:

    \begin{fact}\label{F: group con} Let $S_0$, $S_1$, and $S_2$ be non-trivial strongly minimal plane curves in $\mathcal M$, with canonical bases $s_0$, $s_1$, and $s_2$. Assume that $S_2$ is almost contained in $S_1\circ S_0$. Let $A$ be a set of parameters such that $\rk(s_0/A)=\rk(s_1/A)=1$, $\rk(s_0s_1/A)=2$, and $\rk(s_2/A)\leq 1$. Then $\mathcal M$ interprets a strongly minimal group.
    \end{fact}
    \begin{proof} Let $(x,z)\in S_2$ be generic over $As_0s_1s_2$. Then there is $y$ with $(x,y)\in S_0$ and $(y,z)\in S_1$. Then $(s_0,s_1,s_2,x,y,z)$ forms a rank 1 group configuration (see \cite{bou}).
    \end{proof}

    We first use Fact \ref{F: group con} to interpret a group when $\mathcal E'$ is not (generically) almost faithful. This is standard, but we sketch the details.

    \begin{lemma}\label{L: group configuration from small comp} Assume that for some generic $(a,c)\in M^2$ over $wb_0$, there are infinitely many $(a',c')\in M^2$ such that $E'_{ac}\cap E'_{a'c'}$ is infinite. Then $\mathcal M$ interprets a strongly minimal group.
    \end{lemma}
    \begin{proof}
        Assuming there are infinitely many such $(a',c')$, one can find such an $(a',c')$ with $\rk(a'c'/wb_0ac)\geq 1$. Since $E'_{ac}\cap E'_{a'c'}$ is infinite, there is a strongly minimal set $S_2$ which is almost contained in both $E'_{ac}$ and $E'_{a'c'}$. Let $s_2=\operatorname{Cb}(S_2)$. Then $s_2\in\acl_{\mathcal M}(wb_0ac)\cap\acl_{\mathcal M}(wb_0a'c')$.

        Now $E'_{ac}$ is a union of finitely many compositions $D_v\circ C_t$ for $t\in T$, $v\in V$. Since $S_2$ is strongly minimal, there exists a pair $(t,v)$ such that $D_v\circ C_t\subset E'_{ac}$ and $D_v\circ C_t$ is almost contained in $D_v\circ C_t$. 

        \begin{claim} $\rk(t/wb_0)=\rk(v/wb_0)=1$ and $\rk(tv/wb_0)=2$.
        \end{claim}
        \begin{proof}
            It follows from Lemma \ref{L: generically good} that $a$ and $t$ are $\mathcal M$-interalgebraic over $b_0$, as are $c$ and $v$. So it suffices to note that $\rk(a/wb_0)=\rk(c/wb_0)=1$ and $\rk(ac/wb_0)=2$, which follows since $(a,b)$ is assumed to be generic in $M^2$ over $wb_0$.
        \end{proof}

        Again by the strong minimality of $S_2$, there are strongly minimal sets $S_0\subset C_t$ and $S_1\subset D_v$, with canonical bases $s_0$ and $s_1$, respectively, such that $S_2$ is almost contained in $S_1\circ S_0$. By Fact \ref{F: cb vs af}, $s_0$ and $t$ are $\mathcal M$-interalgebraic over $wb_0$, as are $s_1$ and $v$. Thus the previous claim equivalently gives:

        \begin{claim} $\rk(s_0/wb_0)=\rk(s_1/wb_0)=1$ and $\rk(s_0s_1/wb_0)=2$.
        \end{claim}

        Finally, we check:

        \begin{claim} $\rk(s_2/wb_0)\leq 1$.
        \end{claim}
        \begin{proof} Assume $\rk(s_2/wb_0)\geq 2$. Since $s_2\in\acl_{\mathcal M}(wb_0ac)$ and $\rk(a'c'/wb_0ac)\geq 1$, we get $\rk(a'b'/wb_0s_2)\geq 1$. Then combining with $\rk(s_2/wb_0)\geq 2$ gets $\rk(a'b's_2/wb_0)\geq 3$. But since $s_2\in\acl_{\mathcal M}(wb_0a'c')$, this forces $\rk(a'c'/wb_0)\geq 3$, which is a contradiction since $a',c'\in M$.
        \end{proof}
        Finally, by the previous two claims, we have verified the hypotheses of Fact \ref{F: group con}. We conclude that $\mathcal M$ interprets a strongly minimal group, proving Lemma \ref{L: group configuration from small comp}.
    \end{proof}

    By Lemma \ref{L: group configuration from small comp}, we assume for the rest of the proof that each generic $E'_{ac}$ has infinite intersection with only finitely many $E'_{a'c'}$. By compactness, after deleting a non-generic subset of $E'$ of the form $M^2\times Z_0$ (defined without additional parameters), one obtains the graph of an almost faithful family of plane curves indexed by a large subset of $M^2$. Then by Fact \ref{F: good families}, there are non-generic sets $Z_1,Z_2\subset M^2$ (also defined without additional parameters) so that deleting $(Z_1\times M^2)\cup (M^2\times Z_2)$ yields the graph of a dualizable family. In particular, if $(x,z,a,c)\in E'$ and $(x,z),(a,c)$ are each generic in $M^2$ over $wb_0$, then the deleted set avoids a neighborhood of $(x,z),(a,c)$. Thus, the following is justified:

    \begin{convention}\label{fixed rank 2 family} For the rest of the proof, we fix a dualizable family $\mathcal E=\{E_y:y\in Y\}$ of plane curves which is $\mathcal M(wb_0)$-definable and whose graph $E$ is a large subset of $E'$. Moreover, we assume that if $(x,z,a,c)\in E'$ with $(x,z)$ and $(a,c)$ each generic in $M^2$ over $wb_0$, the sets $E'$ and $E$ agree in a neighborhood of $(x,z,a,c)$.
    \end{convention}

    We will show that $\mathcal E$ has the strong identity property and is not self-dual. This will involve a detailed computational analysis of the construction of $E'$. In order to express the computations smoothly, it will be convenient to introduce the following notions:

    \begin{definition}\label{D: configurations} Let $(a,b,c,x,y,z,t,u)\in M^6\times T\times U$.
    \begin{enumerate} 
    \item We say that $(a,b,c,x,y,z,t,u)$ is a $(\mathcal C,\mathcal D)$-configuration if $(a,b),(x,y)\in C_t$ and $(b,c),(y,z)\in D_u$.
    \item We say that $(a,b,c,x,y,z,t,u)$ is a $(\mathcal C,\mathcal D')$-configuration if it is a $(\mathcal C,\mathcal D)$-configuration and moreover $u\in V$.
    \item We say that $(a,b,c,x,y,z,t,u)$ is a $(\mathcal C,\mathcal D,b_0)$-configuration if it is a $(\mathcal C,\mathcal D)$-configuration and moreover $b=b_0$.
    \item We say that $(a,b,c,x,y,z)$ is a $(\mathcal C,\mathcal D',b_0)$-configuration if it is both a $(\mathcal C,\mathcal D')$-configuration and a $(\mathcal C,\mathcal D,b_0)$-configuration.
    \end{enumerate}
    \end{definition}

    Note that the assertion `$(x,z)\in E'_{ac}$' is equivalent to the existence of an extension to a $(\mathcal C,\mathcal D',b_0)$-configuration $(a,b_0,c,x,y,z,t,u)$. Lemma \ref{L: configuration properties} gives some additional basic properties:

    \begin{lemma}\label{L: configuration properties}
        Let $(a,b,c,x,y,z,t,u)$ be a $(\mathcal C,\mathcal D)$-configuration.
        \begin{enumerate}
            \item Any of $(a,b,t)$ is $\mathcal M$-algebraic over the other two. The same holds of $(x,y,t)$.
            \item $b$ and $c$ are $\mathcal M$-interalgebraic over $u$, as are $y$ and $z$.
            \item $\rk(abcxyztu)\leq 5$, and if equality holds then $u$ is $\mathcal M$-algebraic over $bcyz$.
        \end{enumerate}
    \end{lemma}
    \begin{proof} (1) and (2) are easy consequences of Lemma \ref{L: generically good}. For (3), the inequality $\rk(abcxyztu)\leq 5$ is trivial by (1) and (2), since the whole tuple is $\mathcal M$-algebraic over $bytu$ and by definition $\rk(bytu)\leq 5$. 

    Now suppose $\rk(abcxyztu)=5$. As above it follows that $\rk(bytu)=5$, and since $t\in M$ this implies that $\rk(byu)\geq 4$. In particular, $\rk(bcyzu)\geq 4$. Now since $\mathcal D$ is good we have $\rk(u/bc)\leq 1$, implying that $\rk(bcyz)\geq 3$. In particular, $bc$ and $yz$ are not $\mathcal M$-interalgebraic. Since $\mathcal D$ is good the dual $\mathcal D^\vee$ is almost faithful, so this implies that $D^\vee_{bc}\cap D^\vee_{yz}$ is finite, and in particular $u\in\acl_{\mathcal M}(bcyz)$.
    \end{proof}

    We are now ready to show:
    
    \begin{lemma}\label{L: the family has identity property} $\mathcal E$ has the strong identity property.
    \end{lemma}
    \begin{proof}
        Fix $(a_0,c_0)\in M^2$ generic over $wb_0$. We will work throughout in small neighborhoods of $(a_0,c_0,a_0,c_0)\in E'$. So by Convention \ref{fixed rank 2 family}, we will tacitly disregard the difference between $E$ and $E'$. 
        
        By Lemma \ref{L: generically good}, there are $t_0\in T$ with $(a_0,b_0)\in C_{t_0}$ and $v_0\in V$ with $(b_0,c_0)\in D_{v_0}$. By Lemma \ref{L: generically good}, $(a_0,b_0,t_0)$ and $(b_0,c_0,v_0)$ are generic points of $C$ and $D'$ over $w$, respectively.

        Notice that $(a_0,b_0,c_0,a_0,b_0,c_0,t_0,v_0)$ is a $(\mathcal C,\mathcal D',b_0)$-configuration. Our goal is to use the generic local loric homeomorphism property (Lemma \ref{L: strong local homeo}) to construct a continuous family of $(\mathcal C,\mathcal D,b_0)$-configurations near $(a_0,b_0,c_0,a_0,b_0,c_0,t_0,u_0)$, and then read off a branch of $E$ witnessing the strong identity property.
        
        First, by Lemma \ref{L: generically good}, the genericity of $(a_0,b_0,t_0)\in C$ and $(b_0,c_0,v_0)\in D'$, and Lemma \ref{L: strong local homeo}, each of the following projections restricts to a loric homeomorphism of neighborhoods of the given points:
        
        \begin{itemize}
            \item $C\rightarrow M^2$ at $(a_0,b_0,t_0)\mapsto(a_0,b_0)$.
            \item $D'\rightarrow M^2$ at $(b_0,c_0,v_0)\mapsto(b_0,c_0)$.
            \item $C\rightarrow M^2$ at $(a_0,b_0,t_0)\mapsto(a_0,t_0)$.
            \item $D'\rightarrow M^2$ at $(b_0,c_0,v_0)\mapsto(b_0,v_0)$.
        \end{itemize}
        
        Thus, after applying various shrinkings (and using the restriction property (Lemma \ref{L: strong basic}(\ref{L: strong rest}))), one can choose the following:
        
        \begin{itemize}
            \item $\mathcal R$-definable neighborhoods $X$ of $a_0$ in $M$, $Y$ of $b_0$ in $M$, $Z$ of $c_0$ in $M$, $T'$ of $t_0$ in $T$, and $V'$ of $v_0$ in $V$.
            \item loric maps $f:X\times Y\rightarrow T$, $g:Y\times Z\rightarrow V$, $i:T'\times X\rightarrow Y$, and $j:V'\times Y\rightarrow Z$,
        \end{itemize} so that the following hold:
        
        \begin{enumerate}
            \item $f(X\times Y)\subset T'$ and $g(Y\times Z)\subset V'$.
            \item If $(a,b)\in X\times Y$, and $t=f(a,b)$ then $(a,b)\in C_t$.
		\item If $(b,c)\in Y\times Z$, and $v=g(b,c)$ then $(b,c)\in D_v$.
            \item If $(t,a)\in T'\times X$ and $b=i(t,a)$ then $(a,b)\in C_t$.
    	\item If $(v,b)\in V'\times Y$ and $c=j(v,b)$ then $(b,c)\in D_v$.
		\item $f(a_0,b_0)=t_0$.
            \item $g(b_0,c_0)=v_0$.
            \item $i(t_0,a_0)=b_0$.
            \item $j(v_0,b_0)=c_0$.
                \end{enumerate}

        Now suppose we are given $(a,c,x)\in X\times Z\times X$. Then we can extend $(a,c,x)$ to a $(\mathcal C,\mathcal D',b_0)$-configuration using the above maps: first let $t=f(a,b_0)$ and $u=g(b_0,c)$; then let $y=i(t,x)$ and $z=j(v,y)$. Then the resulting tuple $(a,b_0,c,x,y,z,t,u)$ is a $(\mathcal C,\mathcal D',b_0)$-configuration. Moreover, the corresponding map $(a,c,x)\mapsto(a,b_0,c,x,y,z,t,u)$ on $X\times Z\times X$ is loric; and by (6)-(9) above, we have $(a_0,c_0,a_0)\mapsto(a_0,b_0,c_0,a_0,b_0,c_0,t_0,u_0)$. It follows that the composition $$(a,c,x)\mapsto(a,b_0,c,x,y,z,t,u)\mapsto z$$ witnesses the strong identity property at $(a_0,c_0)$.

        If the reader wishes, we can express the argument formally: for $(a,c,x)\in X\times Z\times X$, set $$h(a,b,x)= j(g(b_0,c),i(f(a,b_0),x)).$$ Using (1)-(9) above and the fact that $E$ and $E'$ agree in a neighborhood of $(a_0,c_0,a_0,c_0)$, one checks that (2) and (3) of Definition \ref{D: identity property} hold for $h$. Moreover, by shrinking we can ensure that $X$ and $Z$ are loric $\dim(M)$-manifolds. Finally, $h$ is loric by repeated applications of the coordinate-by-coordinate and composition properties of loric maps (Lemma \ref{L: strong basic}(\ref{L: strong coordinate} and \ref{L: strong composition})).
    \end{proof}

    To complete the proof of Theorem \ref{T: family exists}, we show that $\mathcal E$ is not self-dual. The main idea is that the curves in $\mathcal E$ come from a fixed $\mathcal M(\emptyset)$-definable family of rank 3 -- while the curves in $\mathcal E^{\vee}$ are tied to the choice of $V$ and $w$. This is why we chose $V$ and $w$ so that $\rk(w)\geq 7$. In fact, we show that strongly minimal subsets of generic curves in $\mathcal E^{\vee}$ have canonical bases of rank at least $\rk(w)-3\geq 4$, which shows that no such curve can also be found in $\mathcal E$. 
    
    \begin{proposition}\label{P: not self dual} $\mathcal E$ is not self-dual.
		\end{proposition}
		\begin{proof} Throughout, we let $E'^{\vee}$ be the set $E'$ with the variables switched, exactly as with families of plane curves. For $(x,z)\in M^2$ we let $E'^{\vee}_{xz}$ be the set of $(a,c)$ with $(a,c,x,z)\in E'^{\vee}$, or equivalently $(x,z,a,c)\in E'$.
  
        Now let $(x_0,z_0)\in M^2$ be generic over $wb_0$. Since $E\subset E'$, it suffices to show that there is no $(a,c)\in M^2$ with $E'^\vee_{x_0z_0}\cap E'_{ac}$ infinite. This follows from the ensuing two lemmas:
			\begin{lemma}\label{L: original curves at most 3} Let $(a,c)\in M^2$ be arbitrary, and let $S$ be strongly minimal and almost contained in $E'_{ac}$. Then $\rk(\cb(S))\leq 3$.
			\end{lemma}
		\begin{proof} $E'_{ac}$ is by definition a union of finitely many compositions $D_v\circ C_t$. Since $S$ is strongly minimal, there is a pair $(v,t)$ such that $S$ is almost contained in $D_v\cap C_t$. So $\cb(S)$ is $\mathcal M$-algebraic over $(v,t)$. Since $v\in U$, $t\in T$, and $U$ and $T$ are (in $\mathcal M$) $\emptyset$-definable sets of ranks 1 and 2, respectively, the lemma follows.
		\end{proof}
	\begin{lemma}\label{L: dual curves at least 4} Let $S$ be strongly minimal and almost contained in $E'^{\vee}_{x_0z_0}$. Then $\rk(\cb(S))\geq 4$.
	\end{lemma}
	\begin{proof} After editing finitely many points, assume $S$ is $\mathcal M$-definable over $s=\cb(S)$. Let $L$ be the set of $u\in U$ such that $(x,z,u)$ extends to a $(\mathcal C,\mathcal D,b_0)$ configuration $(a,b_0,c,x_0,y,z_0,t,u)$ with $(a,c)\in S$. So $L$ is $\mathcal M(b_0x_0z_0s)$-definable. 
    \begin{claim}\label{C: L at most 1} $\rk(L)\leq 1$.
    \end{claim}
    \begin{proof}
        If not, there is $u_0\in L$ with $\rk(u_0/b_0x_0z_0s)\geq 2$. By construction $\rk(b_0x_0z_0)=3$, so then $\rk(u_0b_0x_0z_0)\geq 5$. By definition there is a $(\mathcal C,\mathcal D,b_0)$-configuration $(a,b_0,c,x_0,y,z_0,t,u_0)$ with $(a,b)\in S$. Then $\rk(ab_0cx_0yz_0tu_0)\geq 5$, so by Lemma \ref{L: configuration properties}(3), we get $u_0\in\acl_{\mathcal M}(b_0cyz_0)$. 
        
        On the other hand, since $(a,c)\in S$ we have $\rk(ac/s)\leq 1$, so $\rk(ac/b_0x_0z_0s)\leq 1$; and by Lemma \ref{L: configuration properties}(1) and the previous paragraph, we have $t\in\acl_{\mathcal M}(ab_0)$, $y\in\acl_{\mathcal M}(x_0t)$, and $u_0\in\acl_{\mathcal M}(b_0cyz_0)$ -- thus $u_0\in\acl_{\mathcal M}(ab_0cx_0z_0s)$. So combining these gets $\rk(u_0/b_0x_0z_0s)\leq 1$, a contradiction.
    \end{proof}
    \begin{claim}\label{C: L almost contains V} $L$ almost contains $V$.
    \end{claim}
    \begin{proof}
        Let $(a,c)\in S$ be generic over $wb_0x_0z_0s$. Since $S$ is almost contained in $E'^\vee_{x_0z_0}$, we get that $(a,c)\in E'^\vee_{x_0z_0}$. So there is an extension to a $(\mathcal C,\mathcal D',b_0)$-configuration $(a,b_0,c,x_0,y,z_0,t,u)$. Note that $c\in\acl_{\mathcal M}(ub_0)$, by Lemma \ref{L: configuration properties}(2). Moreover, since $S\subset E'^\vee_{x_0z_0}$, clearly $S$ is a non-trivial plane curve, implying that $a\in\acl_{\mathcal M}(sc)$. So in total, we have $(a,c)\in\acl_{\mathcal M}(wb_0x_0z_0u)$. But by assumption $\rk(ac/wb_0x_0z_0)\geq 1$, so $ac$ has different ranks over $wb_0x_0z_0$ and $wb_0x_0z_0u$. This implies that $u\notin\acl_{\mathcal M}(wb_0x_0z_0)$, so that $\rk(u/wb_0x_0z_0)\geq 1$. But $u\in L\cap V$ by definition, and $L\cap V$ is $\mathcal M(wb_0x_0z_0)$-definable. Thus $\rk(L\cap V)\geq 1$, and since $V$ is strongly minimal, the claim follows. 
    \end{proof}
    
        By Claims \ref{C: L at most 1} and \ref{C: L almost contains V}, $V$ is (up to almost equality) a stationary component of $L$; so $w$ is $\mathcal M$-algebraic over $b_0x_0z_0s$, and thus $\rk(b_0x_0z_0s)\geq\rk(w)\geq 7$. Since $b_0,x_0,z_0\in M$, this implies $\rk(s)\geq 7-3=4$, proving Lemma \ref{L: dual curves at least 4}.
	\end{proof}
        To conclude the proof of Proposition \ref{P: not self dual}, note that if $E'_{ac}\cap E'^\vee_{x_0y_0}$ is infinite for some $(a,c)$, then there is a strongly minimal set almost contained in both $E'_{ac}$ and $E'^\vee_{x_0z_0}$. By Lemmas \ref{L: original curves at most 3} and \ref{L: dual curves at least 4}, this is impossible.
	\end{proof}
        Finally, by Lemma \ref{L: the family has identity property} and Proposition \ref{P: not self dual}, the proof of Theorem \ref{T: family exists} is complete.
    \end{proof}
    
    \section{Proof of the Main Theorem}\label{S: main thm}
        We now drop Assumption \ref{A: R and M}, and conclude our main theorem:
        \begin{theorem}[Theorem \ref{T: intro}] Let $\mathcal R$ be an o-minimal expansion of a field. Let $\mathcal M$ be a non-locally modular strongly minimal $\mathcal R$-relic. Then $\dim(M)=2$.
        \end{theorem}
        \begin{proof} It was shown in \cite{HaOnPe} that $\dim(M)\geq 2$; we show that $\dim(M)\leq 2$. 
        
        By elimination of imaginaries in $\mathcal R$, we may assume $M\subset R^n$ for some $n$. After passing to an elementary extension, restricting to a finite sublanguage of $\mathcal L(\mathcal M)$ witnessing local modularity, and adding a countable set of constants, we may assume $\mathcal R$ and $\mathcal M$ are as in Assumption \ref{A: R and M}, where $\mathcal S$ is the full lore and $d_0=2$ (here we use Proposition \ref{ramification prop} to get manifold ramification purity of order 2).
        
        If $\mathcal M$ does not interpret a strongly minimal group, then Theorem \ref{T: family exists} gives a dualizable family of plane curves that has the strong identity property and is not self-dual -- thus Theorem \ref{T: general case from family} gives $\dim(M)\leq 2$.
        
        Otherwise, assume $\mathcal M$ interprets the strongly minimal group $(G,+)$. By strong minimality, $G$ and $M$ are in finite correspondence, so it suffices to show that $\dim(G)\leq 2$. 
        
        Let $\mathcal G$ be the induced structure on $G$ from $\mathcal M$. Then $\mathcal G$ is strongly minimal and not locally modular. By \cite{Pi5} and \cite[Chapter 10, Theorem 1.8]{vdDries}, $G$ can be $\mathcal R$-definably embedded into some $R^m$ so that $(G,+)$ is a topological group with the subspace topology. By \cite[Proposition 6.1]{omeuler}), $(G,+)$ has finite 2-torsion. Thus, $\mathcal G$ satisfies Assumption \ref{A: group ass}, and Theorem \ref{T: group case} gives $\dim(G)\leq 2$. 
        \end{proof}

        \section{Relics of Compact Complex Manifolds}\label{S: CCM}

        The multi-sorted structure $\mathcal A$ of compact complex manifolds (CCM) consists of all compact complex manifolds (as distinct sorts), equipped with all analytic subsets of their finite products. $\mathcal A$ has quantifier elimination and elimination of imaginaries, and each sort of $\mathcal A$ has finite Morley rank. The model-theoretic study of $\mathcal A$ was initiated by Zilber (\cite{zilbereaster}); see \cite{moosa2010} for a brief survey.
        
        It is well-known that $\mathcal A$ is interpretable in the o-minimal structure $\mathbb R^{an}$ (the real field with all analytic functions restricted to compact domains). Thus, we can view $\mathcal A$ as an $\mathbb R^{an}$-relic, and Theorem \ref{T: intro} applies. In this section, we use Theorem \ref{T: intro} to give a short proof of the Zilber trichotomy for all relics of CCM. Note that the standard interpretation of $\mathcal A$ in $\mathbb R^{an}$ is 2-dimensional (as expected): if $X$ is a compact complex manifold of complex dimension $d$, then the interpretation identifies $X$ with an $\mathbb R^{an}$-definable set of o-minimal dimension $2d$.
        
        Note that $\mathcal A$ interprets a canonical copy of the complex field (by deleting a point from the projective line). Similarly, every model $\mathcal A'=\operatorname{Th}(\mathcal A)$ defines a canonical algebraically closed field (with the same definition). Let us call this field $K_{\mathcal A'}$. We will show that every non-locally modular strongly minimal relic of any $\mathcal A'\models\operatorname{Th}(\mathcal A)$ interprets a copy of $\mathcal K_{\mathcal A'}$.
               
        Other than Theorem \ref{T: intro}, our main tool is Moosa's non-standard Riemann existence theorem (\cite{MooRE}). This theorem will allow us (in some cases) to reduce the trichotomy over CCM to trichotomy over ACF.
        
        Let $\mathcal A'\models\operatorname{Th}(\mathcal A)$. By quantifier elimination, every $\mathcal A'$-definable set can be endowed with a `complex' dimension, which we denote $\dim_{CCM}(X)$. 
        
        \begin{fact}\label{F: Moosa Riemann}\cite{MooRE} Let $\mathcal A'\models\operatorname{Th}(\mathcal A)$. Let $X$ be $\mathcal A'$-definable with $\dim_{CCM}(X)=1$. Then $X$ is algebraic: $X$ is $\mathcal A'$-definably isomorphic to a one-dimensional constructible set over $K_{\mathcal A'}$ with only algebraic structure.
        \end{fact}
        
        We are now ready to give the proof.
        
        \begin{theorem}\label{T: CCM} Let $\mathcal A'\models\operatorname{Th}(\mathcal A)$. Let $\mathcal M$ be a non-locally modular strongly minimal $\mathcal A'$-relic. Then $\mathcal M$ interprets a field $\mathcal A'$-definably isomorphic to $K_{\mathcal A'}$.
        \end{theorem}
        \begin{proof} By standard and easy reductions, we may assume $\mathcal A'$ and $\mathcal M$ are sufficiently saturated (see the proof of Theorem 9.9 of \cite{CasACF0}).
        
       Passing to an elementary extension if necessary, we may view $\mathcal A'$ as an $\mathcal R$-relic for some $\mathcal R\models\operatorname{Th}(\mathbb R^{an})$ (using the same formulas that interpret $\mathcal A$ in $\mathbb R^{an}$ -- thus the interpretation of $\mathcal A'$ in $\mathcal R$ is also 2-dimensional). 
       
       Notice that $\mathcal M$ is both an $\mathcal A'$-relic and an $\mathcal R$-relic. By Theorem \ref{T: intro}, $\dim_{\mathcal R}(M)=2$. Thus $\dim_{CCM}(M)=1$. So by Fact \ref{F: Moosa Riemann}, $\mathcal M$ is isomorphic to a $(\mathcal K_{\mathcal A'},+,\cdot)$-relic. Now apply the trichotomy for ACF$_0$-relics (any of \cite{CasACF0}, \cite{HaSu}, and \cite{CaHaYe} apply here).
        \end{proof}
     
        \begin{remark} By standard arguments, Theorem \ref{T: CCM} implies that an arbitrary $\mathcal A'$-relic is 1-based or interprets a copy of $K_{\mathcal A'}$. This uses only that $\mathcal A$ has finite Morley rank, and works as in \cite[Theorem 9.9]{CasACF0}.
        \end{remark}

        \appendix
        \section{Some O-minimal Algebraic Topology}
        The proof of Theorem \ref{T: intro} crucially used that o-minimal expansions of fields satisfy manifold ramification purity of order 2 (Definition \ref{D: ramification purity}). This is really a fact in o-minimal algebraic topology, and otherwise has nothing to do with the rest of the paper. We now give a proof in the appendices. Appendix A gives some preliminary lemmas, and Appendix B contains the main argument.

        \begin{assumption}
            \textbf{Throughout the appendices, we fix $\mathcal R$, an o-minimal expansion of a real closed field}. All model-theoretic terminology (\textit{definable, dimension}, etc.) is now interpreted in the sense of $\mathcal R$.
        \end{assumption}

        \subsection{Invariance of Domain}

        Recall that a \textit{definable $n$-manifold} is a definable set equipped with a finite definable atlas, i.e. a finite open cover by sets definably homeomorphic to $R^n$. We first recall the definable version of the Invariance of Domain Theorem:

        \begin{fact}[\cite{pie}]\label{invariance of domain} Let $f:X\rightarrow Y$ be a definable, injective, continuous map of definable $n$-manifolds. Then $f$ is an open map.
        \end{fact}

        This is proven in the case $X=Y=R^n$ in \cite{pie}, and the case of arbitrary definable manifolds follows because openness can be checked locally.

        \subsection{Connectedness Notions and Removing a Small Set}
      
         Working with definable manifolds, one can give definable analogs of various topological notions (e.g. paths, homotopies, fundamental groups, covering maps, homology, etc.). For most of these notions, we will not give more details here; see e.g. \cite{edmote} or \cite{BerOte} for a more thorough account. 

        Let $X$ be a definable manifold. $X$ is \textit{definably connected} if any two points in $X$ can be joined by a \textit{definable path} (i.e. a continuous definable map from the interval $[0,1]$ in $R$). $X$ is \textit{definably simply connected} if $X$ is definably connected with trivial definable fundamental group (i.e. every definable loop in $X$ admits a definable homotopy with the trivial loop). As expected, $R^n$ is definably simply connected for all $n$.

         We will use that definable connectedness and definable simple connectedness are unchanged by deleting subsets of appropriately small dimension. These facts are likely known, but we did not see them in the literature, so we give the proofs (they are also in the author's PhD thesis).

         \begin{lemma}\label{manifold stays connected} Let $T\subset R^n$ be definable.
		\begin{enumerate}
			\item If $\dim(T)\leq n-2$ then $R^n-T$ is definably connected.
			\item If $\dim(T)\leq n-3$ then $R^n-T$ is definably simply connected. 
		\end{enumerate}
	\end{lemma}
        \begin{proof}
        \begin{enumerate}
            \item Let $a,b\in R^n-T$. Let $\gamma:[0,1]\rightarrow R^n$ be a definable path connecting $a$ and $b$ in $R^n$ (e.g. a straight line will do). We show that $\gamma$ can be modified to avoid $T$.

            Let $f:[0,1]\rightarrow R$ be continuous and definable such that $f(x)=0$ if and only if $x=0,1$ (e.g. take $f(x)=x(1-x)$). For each $v\in R^n$, define a path $\gamma_v$ from $a$ to $b$ by $\gamma_v(x)=\gamma(x)+f(x)\cdot v$. Let $S$ be the set of $v$ such that $\gamma_v$ meets $T$. We show that $\dim(S)<n$, and so there is some $v\in R^n-S$ (i.e. some $\gamma_v$ avoids $T$).

            Let $B$ be the set of $(v,x,t)$ with $\gamma_v(x)=t\in T$. If $v\in S$ then $v$ extends to an element of $B$. On the other hand, each $(x,t)\in[0,1]\times T$ extends to at most one element of $B$: given $(v,x,t)\in B$ we have $v=\frac{1}{f(x)}\cdot(t-\gamma(x))$, so $v$ is determined by $x$ and $t$.
            
            Thus, we have $$\dim(S)\leq\dim(B)\leq\dim([0,1]\times T)=\dim(T)+1\leq(n-2)+1<n,$$ as desired.
            
            \item Similar to (a). First apply (a) to see that $R^n-T$ is definably connected. Now let $\gamma$ be a definable loop at $a$ in $R^n-T$. Let $h:[0,1]^2\rightarrow R^n$ be a definable homotopy to the trivial loop in $R^n$ (so $h(0,-)$ is $\gamma$ and $h(1,-)$ is constant at $a$). We modify $h$ to avoid $T$.

            Let $f$ be as in (a). For $v\in R^n$, define $h_v:[0,1]^2\rightarrow R^n$ by $h_v(x,y)=h(x,y)+f(x)\cdot v$. Let $S$ be the set of $v$ with $h_v$ meeting $T$, and let $B$ be the set of $(v,x,y,t)$ with $h_v(x,y)=t\in T$. As in (a), if $v\in S$ then $v$ extends to an element of $B$; and each $(x,y,t)$ determines at most one element of $B$. Thus, we have $$\dim(S)\leq\dim(B)\leq\dim([0,1]^2\times T)=\dim(T)+2\leq(n-3)+2<n,$$ and so there is some $v\in R^n-S$ (i.e. some $h_v$ avoiding $T$).            
        \end{enumerate}
        \end{proof}

        \begin{remark} Lemma \ref{manifold stays connected}(1) (resp. (2)) should remain true with $R^n$ replaced by any definably connected (resp. definably simply connected) definable $n$-manifold. For (1), this is straightforward: one uses o-minimality and a finite atlas to break $\gamma$ into finitely many segments, each of which is contained in an open copy of $R^n$; one then handles each segment as in the proof above. For (2), things are less clear, and it would be nice to see a more general proof.
        \end{remark}

        \subsection{Definable Coverings}

        Edmundo and Ortero (\cite{edmote}) developed the Galois correspondence for definable covering maps. Their definition of definable covers requires a `finite atlas', similar to definable manifolds. Unfortunately, we need the Galois correspondence without assuming a finite atlas. In this section, we prove (thankfully) that the existence of a finite atlas is automatic. Let us be more precise:

        \begin{definition}\label{D: definable cover} Let $f:X\rightarrow Y$ be a continuous, surjective, definable map of definable sets.
        \begin{enumerate}
     \item If $U\subset Y$ is definable and open, we say that $f$ \textit{definably trivializes} over $U$ if $f^{-1}(U)$ is the disjoint union of finitely many definable subsets $S_1,...,S_k$ such that the restriction of $f$ to each $S_j$ is a homeomorphism with $U$.
     \item A \textit{trivializing atlas} for $f$ is an open cover $\{U_i:i\in I\}$ of $Y$ by definable open sets over which $f$ definably trivializes. If $I$ is finite, then $\{U_i\}$ is a \textit{finite trivializing atlas}.
     \item We say that $f$ is a \textit{weak definable cover} if it admits a trivializing atlas, and that $f$ is a \textit{definable cover} if it admits a finite trivializing atlas.
     \end{enumerate}
     \end{definition}

     Fact \ref{F: galois correspondence} is (essentially) proven in \cite{edmote}. Our contribution is Proposition \ref{P: covering equivalence}. The proposition should probably be already known, but it seems not to appear in the literature. Our proof is equivalent to one from the author's PhD thesis.

     \begin{fact}\label{F: galois correspondence} Let $X$ and $Y$ be definable $n$-manifolds. Assume that $X$ is definably connected and $Y$ is definably simply connected. Let $f:X\rightarrow Y$ be a definable cover. Then $f$ is injective.
     \end{fact}
        \begin{proof} This follows from the Galois correspondence, which is developed in Section 2 of \cite{edmote} (precisely, see Corollary 2.8 and Proposition 2.10).
        \end{proof}

        \begin{proposition}\label{P: covering equivalence}
            Let $X$ and $Y$ be definable sets, and let $f:X\rightarrow Y$ be a weak definable cover. Then $f$ is a definable cover. In particular, Fact \ref{F: galois correspondence} holds for weak definable covers. 
        \end{proposition}
        \begin{proof} First, we can assume all fibers of $f$ have the same size: by elimination of $\exists^\infty$ in $\mathcal R$, only finitely many distinct fiber sizes occur; and by the definition of trivialization, each size occurs on an open set. So we can separately handle the restriction of $f$ to the fibers of each size.
        
        So, assume all fibers of $f$ have size $k>0$. Say that a \textit{preimage enumerator} is a definable function $g:Y\rightarrow X^k$ such that each $g(y)$ enumerates $f^{-1}(y)$ in some order. Note that if $U\subset Y$ is definable and open, and some preimage enumerator is continuous on $U$, then $f$ definably trivializes over $U$. We will find finitely many preimage enumerators $g_i$, and corresponding definable open sets $U_i\subset Y$, so that each $g_i$ is continuous on $U_i$ and the $U_i$ cover $Y$.

        Suppose $X\subset R^n$, and fix any definable linear ordering $<_0$ on $R^n$ (for example, the lexicographic order).

        \begin{notation}
            Let $T:R^n\rightarrow R$ be any continuous definable map.
            \begin{enumerate}
                \item By $<_T$, we mean the definable linear ordering on $R^n$ given by $x<y$ if $T(x)<T(y)$, with the fixed order $<_0$ used as a tiebreaker.
                \item By $g_T$, we mean the preimage enumerator which enumerates the preimage of each $y\in Y$ in $<_T$-increasing order.
                \item By $\operatorname{Split}(T)$, we mean the \textit{splitting locus} of $T$: the set of $y\in Y$ such that $T$ is injective on $f^{-1}(y)$.  
            \end{enumerate}
        \end{notation}

        The following is clear from the continuity of $f$:
        
        \begin{claim} Let $T:R^n\rightarrow R$ be any definable continuous map. Then $\operatorname{Split}(T)$ is open in $Y$, and $g_T$ is continuous on $\operatorname{Split}(T)$.
        \end{claim}

        To prove Proposition \ref{P: covering equivalence}, we will find finitely many definable continuous maps $T_i:R^n\rightarrow R$ so that the corresponding open sets $\operatorname{Split}(T_i)$ cover $Y$.
        
        In fact, we show that linear maps suffice. Let $\mathcal T$ be the collection of all $R$-linear maps $T:R^n\rightarrow R$. We can view $\mathcal T$ as a definable set of dimension $n$, identifying each $T\in\mathcal T$ as a row vector.

        \begin{claim} There are finitely many $T_i\in\mathcal T$ such that the sets $\operatorname{Split}(T_i)$ cover $Y$.
        \end{claim}
        \begin{proof} It suffices to work in an $\aleph_1$-saturated elementary extension of $\mathcal R$: if we find such maps $T_1,...,T_m$ in an elementary extension, then there is a first-order statement saying that $m$ maps suffice to prove the claim, and this statement also holds in $\mathcal R$.

        So, assume $\mathcal R$ is $\aleph_1$-saturated. Absorbing parameters, assume $f:X\rightarrow Y$ is $\emptyset$-definable. Let $I$ be a countably infinite set, and let $\{T_i:i\in I\}$ be independent generics of $T$. By $\aleph_1$-saturation, it suffices to show that the sets $\operatorname{Split}(T_i)$ cover $Y$.

        Let $y\in Y$, and suppose $y\notin\operatorname{Split}(T_i)$ for all $i$. By pigeonhole, there are $x\neq x'\in f^{-1}(y)$, and an infinite $J\subset I$, so that $T_j(x)=T_j(x')$ for all $j\in J$. So for each $j\in J$, $T_j$ is non-generic in $\mathcal T$ over $xx'$.
        
        Enumerate $J=(j_1,j_2,...)$. Then for each $N$ we get $$\dim(T_{j_1}...T_{j_N}/xx')\leq N\cdot(n-1)=\dim(T_{j_1}...T_{j_N})-N,$$ and thus $\dim(xx')\geq N$. As $xx'$ is fixed and $N$ is arbitrary, this is a contradiction. 
        \end{proof}
        By the claim, we have proven Proposition \ref{P: covering equivalence}.
        \end{proof}

\section{O-minimal Purity of Ramification}

We are now ready to prove purity of ramification. Recall that for a function $f:X\rightarrow Y$ of topological spaces, $\operatorname{Ram}(f)$ is the set of $x\in X$ such that $f$ is non-injective on every neighborhood of $x$.

\begin{proposition}\label{ramification prop} Let $X$ and $Y$ be definable $n$-manifolds, and let $f:X\rightarrow Y$ be a finite-to-one, continuous, definable function. Then either $\operatorname{Ram}(f)=\emptyset$, or $\dim(\operatorname{Ram}(f))\geq n-2$.
		\end{proposition}

  \begin{proof} For shorthand, let us denote $\operatorname{Ram}(f)$ by $N$ and its complement in $X$ by $I$ ($I$ stands for `injective' and $N$ stands for `not injective'). Note that $I$ is open in $X$: if $x\in I$, then $f$ is injective on some open neighborhood $U$ of $x$, and thus each $x'\in U$ belongs to $I$ (also witnessed by $U$). So $X$ is the disjoint union of the open set $I$ and the closed set $N$. 
  
  \begin{assumption}
      \textbf{Throughout the proof of Proposition \ref{ramification prop}, assume that $\dim(N)\leq n-3$.}
    \end{assumption}
    
    We will prove Proposition \ref{ramification prop} by proving that $N=\emptyset$, i.e. $f$ is locally injective everywhere.
	
	\begin{claim}\label{locally injective implies open} The restriction of $f$ to $I$ is an open map.
	\end{claim}
	\begin{proof} Since $I$ is open in $X$, it is also a definable $n$-manifold. Then, since openness can be checked locally, the claim follows by invariance of domain (Fact \ref{invariance of domain}).
	\end{proof}

    Our main tool will be \textit{good pairs} (Definition \ref{good neighborhood pair}):

    \begin{notation} We use $\partial$ to denote topological boundary, i.e. $\partial(X)$ is the closure of $X$ minus the interior of $X$. 
    \end{notation}
    
		\begin{definition}\label{good neighborhood pair} Let $x\in X$, and set $y=f(x)\in Y$. By a \textit{good pair} at $x$, we mean a pair of definable open neighborhoods $U\subset X$ of $x$ and $V\subset Y$ of $y$ satisfying each of the following:
			\begin{enumerate}
				\item $\overline U$ is definably compact and contains no elements of $f^{-1}(y)$ other than $x$.
				\item $V$ is definably homeomorphic to $R^n$.
				\item $f(\partial (U))\cap V=\emptyset$.
				\end{enumerate}
			\end{definition}
		
	\begin{claim}[Existence of Good Pairs]\label{good pairs exist} For each $x\in X$, there is a neighborhood basis of $(x,f(x))$, in $X\times Y$, consisting of sets $U\times V$ such that $(U,V)$ is a good pair at $x$.
		\end{claim}
	\begin{proof} Set $y=f(x)$. Since $X$ and $Y$ are definable $n$-manifolds, we may assume that $X=Y=R^n$. Let $U_0$ be any open neighborhood of $x$, and $V_0$ any open neighborhood of $y$. We will find a good pair $(U,V)$ at $x$ with $U\subset U_0$ and $V\subset V_0$. 
		
		Let $U$ be any bounded open ball around $x$ which is contained in $U_0$. Since $f$ is finite-to-one, the fiber $f^{-1}(y)$ is finite, so we may assume that $\overline U$ is disjoint from $f^{-1}(y)-\{x\}$. Combined with the boundedness of $U$, we attain (1) from Definition \ref{good neighborhood pair}. 
		
		Since $\overline U$ is definably compact, so is $\partial U$, and thus so is $f(\partial (U))$. Since $U$ is a neighborhood of $x$, and $\overline U$ only meets $f^{-1}(y)$ at $x$, it follows that $\partial(U)$ is disjoint from $f^{-1}(y)$, i.e. $y\notin f(\partial(U))$. Since $f(\partial(U))$ is definably compact -- and thus closed -- there is an open ball $V\subset V_0$ at $y$ which is disjoint from $f(\partial U)$. Thus we satisfy (2) and (3) from Definition \ref{good neighborhood pair}, completing the proof. 
		\end{proof}
	
	\begin{remark} The proof of Claim \ref{good pairs exist} shows that whenever $(U,V)$ is a good pair at $x$, the same is true of $(U,V')$ for any open ball $V'\subset V$ at $y$.
		\end{remark}
	
	The next few claims give basic properties of good pairs. The goal is to show that $f$ is open everywhere (Claim \ref{unramified open}).
	
\begin{claim}\label{closed image in good pair} Let $(U,V)$ be a good pair at $x\in X$. Then $f(U)\cap V$ is relatively closed in $V$.
\end{claim}
\begin{proof} By (1) of Definition \ref{good neighborhood pair}, $\overline U$ is definably compact; using this it is easy to check that $\overline{f(U)}=f(\overline U)$. On the other hand, by (3) of Definition \ref{good neighborhood pair} we have that $f(\overline U)\cap V=f(U)\cap V$. In particular, $$f(U)\cap V=\overline{f(U)}\cap V,$$ which shows that $f(U)\cap V$ is relatively closed in $V$.
\end{proof}
\begin{claim}\label{small boundary}
	Let $(U,V)$ be a good pair at $x\in X$. Then the relative boundary of $f(U)\cap V$ in $V$ has dimension at most $n-3$.
\end{claim}
	\begin{proof} Let $B$ be the relative boundary of $f(U)\cap V$ in $V$. By Claim \ref{closed image in good pair}, $B\subset f(U)$. By Claim \ref{locally injective implies open}, $B$ is disjoint from $f(U\cap I)$. Thus $B\subset f(N)$. Since $f$ is finite-to-one, $$\dim (f(N))\leq\dim (N)\leq n-3,$$ which completes the proof.
		\end{proof}
	\begin{claim}\label{good pair locally surjective} Let $(U,V)$ be a good pair at $x\in X$. Then $f(U)\supset V$.
		\end{claim}
	\begin{proof} As in Claim \ref{small boundary}, let $B$ be the relative boundary of $f(U)\cap V$ in $V$. Removing $B$ partitions $V$ into two open subsets: the interior of $f(U)\cap V$, and the complement $V-f(U)$ (the second set is open by Claim \ref{closed image in good pair}). By Definition \ref{good neighborhood pair}, $V$ is definably homeomorphic to $R^n$, so by Claim \ref{small boundary} and Fact \ref{manifold stays connected}, removing $B$ cannot disconnect $V$. Thus, one of $\operatorname{int}(f(U)\cap V)$ and $V-f(U)$ is empty. 
		
		Since $f$ is continuous and finite-to-one, $\dim(f(U)\cap V)=n$. Thus $f(U)\cap V$ has non-empty interior. So $V-f(U)$ is empty, i.e. $f(U)\supset V$.
	\end{proof}

\begin{claim}\label{unramified open} $f$ is an open map.
\end{claim}
\begin{proof} By Claims \ref{good pairs exist} and \ref{good pair locally surjective}.
	\end{proof}

 Recall that we want to show $f$ is locally injective everywhere. For the rest of the proof, fix $x_0\in X$, and set $y_0=f(x_0)$. We will show that in small neighborhoods of $x_0$ and $y_0$, $f$ becomes a definable cover after deleting a codimension 3 set of points. We then use Lemma \ref{manifold stays connected} to show that this cover goes from a definably connected space to a definably simply connected space, and deduce local injectivity from Fact \ref{F: galois correspondence}. 
 
Let us be more precise. Let $(U_0,V_0)$ be any good pair at $x_0$; we will show that $f$ is injective on the neighborhood $U_0\cap f^{-1}(V_0)$ of $x_0$. By Claim \ref{good pair locally surjective}, note that $f(U_0)\supset V_0$.

\begin{notation} For the rest of the proof, we set $$W=\{y\in V_0:U_0\cap f^{-1}(y)\subset I\}.$$
\end{notation}

So $W$ is formed by deleting all fibers that meet $N$. Since $f(U_0)\supset V_0$, we equivalently have $$W=V_0-f(U_0\cap N).$$ The next claims give the key properties of $W$. 

\begin{claim}\label{covering set is small} $W$ is a dense open subset of $V_0$, and $\dim(V_0-W)\leq n-3$.
\end{claim}
\begin{proof} From above, we have $$W=V_0-f(U_0\cap N).$$ By (3) of Definition \ref{good neighborhood pair}, we equivalently have $$W=V_0-f(\overline{U_0}\cap N).$$ Since $\overline{U_0}$ is definably compact, and $N$ is a closed set of dimension at most $n-3$, $\overline{U_0}\cap N$ is both definably compact and of dimension at most $n-3$. Since $f$ is continuous and finite-to-one, we $f(\overline{U_0}\cap N)$ is also definably compact of dimension at most $n-3$. This set is exactly the complement of $W$ in $V_0$, so $W$ is open in $V_0$ and $\dim(V_0-W)\leq n-3$.
	
	Finally, the density of $W$ in $V_0$ is because $\dim(V_0-W)<\dim(V_0)$. To elaborate, let $Z$ be any non-empty open subset of $V_0$. Since $V_0$ is a definable $n$-manifold, we have $\dim (Z)=n$, and thus $\dim(V_0-W)<\dim Z$. So there is an element of $Z$ not belonging to $V_0-W$ -- equivalently, $W\cap Z\neq\emptyset$.
	\end{proof}

\begin{claim} The restriction $f:U_0\cap f^{-1}(W)\rightarrow W$ is a weak definable cover.
	\end{claim}
	\begin{proof} Since $f(U_0)\supset V_0$ (Claim \ref{good pair locally surjective}), the given restriction is automatically surjective.
 
 Fix $y\in W$. Then $U_0\cap f^{-1}(W)$ is a non-empty finite set, say $\{x_1,...,x_m\}$. Let $A_1,...,A_m\subset U_0$ be pairwise disjoint definable open neighborhoods of $x_1,...,x_m$ on which $f$ is injective. By either Fact \ref{invariance of domain} or Claim \ref{unramified open}, the restriction of $f$ to each $A_i$ is a homeomorphism with some open set $B_i\subset N$.
		
	Now the set $C=\overline{U_0}-\bigcup_{i-1}^mA_i$ is definably compact and disjoint from $f^{-1}(y)$; it follows that $f(C)$ is definably compact and does not contain $y$. So, we may choose a definable open neighborhood $Z\subset W$ of $y$ which is disjoint from $f(C)$. Since each $B_i$ is open, we may assume that $Z\subset B_i$ for each $i$. Thus, for each $i$, $f$ gives a homeomorphism from $A_i\cap f^{-1}(Z)$ to $Z$.
	
	On the other hand, since $Z$ is disjoint from $f(C)$, the preimage of $Z$ in $U_0$ is precisely the union of the $A_i\cap f^{-1}(Z)$. In other words, the preimage of $Z$ in $U_0$ consists of $m$ pairwise disjoint homeomorphisms, which proves the claim.
	\end{proof}

We now use the Galois correspondence. By Claim \ref{covering set is small}, $W$ is the complement in $V_0\cong R^n$ of a definable subset of codimension at least 3. So by Proposition \ref{P: covering equivalence} and Lemma \ref{manifold stays connected}(2), we can write $U_0\cap f^{-1}(W)$ as a disjoint union $S_1\cup...\cup S_m$, where each $S_i$ is a definably homeomorphic copy of $W$. 

Our next goal is to show that $m=1$. The idea is that restricting to $f^{-1}(W)$ only (locally) removes a codimension 2 set, so $S_1\cup...\cup S_m$ should still be connected. 

\begin{claim}\label{closure of S} For each $i$, $x_0\in\overline{S_i}$.
\end{claim}
\begin{proof} Since $S_i\subset U_0$, $\overline{S_i}$ is definably compact. This implies that $\overline{f(S_i)}=f(\overline{S_i})$. Now $f(S_i)=W$ is dense in $V_0$ by Claim \ref{covering set is small}, so $y_0\in\overline{f(S_i)}$, and thus $y_0\in f(\overline{S_i})$. Again since $S_i\subset U_0$, and $f^{-1}(y)$ only has one point of $U_0$ (Definition \ref{good neighborhood pair}(1)), the only possibility is that $x_0\in\overline{S_i}$.
	\end{proof}

\begin{claim}\label{injective set connected} $U_0\cap f^{-1}(V_0)$ is a definably connected definable $n$-manifold.
\end{claim}
\begin{proof} That $U_0\cap f^{-1}(V_0)$ is a definable $n$-manifold is clear, since it is open in the definable $n$-manifold $X$. We show that $U_0\cap f^{-1}(V_0)$ is definably connected.
	
Let $x\in U_0\cap f^{-1}(V_0)$, and let $C$ be the connected component of $x$ in $U_0\cap f^{-1}(V_0)$. Since $U_0\cap f^{-1}(V_0)$ is open, so is $C$. So by Claim \ref{unramified open}, $f(C)$ is an open neighborhood of $f(y)\in V_0$. Since $W$ is dense in $V_0$, there is some $y'\in f(C)\cap W$. Let $x'\in C$ with $f(x')=y'$. Since $y'\in W$, we have $x'\in S_i$ for some $i$. But $S_i\cong W$ is definably connected, so by Claim \ref{closure of S} and definable choice it follows that $x'$ and $x_0$ are definably path connected. Since $x,x'\in C$, they are also definably path connected -- thus by concatenating, $x$ and $x_0$ are definably path connected.
	
We have shown that an arbitrary element of $U_0\cap f^{-1}(V_0)$ is definably path connected to $x_0$, which proves the claim. 
\end{proof}
\begin{claim}\label{only one sheet} $m=1$.
	\end{claim}
\begin{proof} Let $T$ be the set of points in $U_0\cap f^{-1}(V_0)$ which do not belong to any $S_i$. In other words, $T=U_0\cap f^{-1}(V_0-W)$. By Claim \ref{covering set is small} and the fact that $f$ is finite-to-one, we have $\dim(T)\leq n-3$. Combined with Claim \ref{injective set connected} and Fact \ref{manifold stays connected}, we conclude that the union of the $S_i$ is definably connected. Since the $S_i$ are pairwise disjoint open sets, the only possibilities are $m=0$ and $m=1$. If $m=0$ then no point in $W$ would have a preimage in $U_0$, contradicting Claim \ref{good pair locally surjective}. So $m=1$. 	
\end{proof}
Finally, we finish the proof of Proposition \ref{ramification prop}:
\begin{claim}\label{conclusion of ramification prop}
	$f$ is injective on $U_0\cap f^{-1}(V_0)$.
\end{claim}
\begin{proof}
	We want to show that each element of $V_0$ has at most one preimage in $U_0$. Fix $y\in V_0$. Then $U_0\cap f^{-1}(y)$ is a finite set, say $\{x_1,...,x_l\}$. For each $i=1,...,l$, choose a good pair $(U_i,V_i)$ at $x_i$. By Claim \ref{good pairs exist}, we may assume that the $U_i$ are pairwise disjoint subsets of $U_0$, and that each $V_i$ is contained in $V_0$.
	
	Let $V=\bigcap_{i=1}^mV_i$. So $V\subset V_0$ is an open neighborhood of $y$ in $N$. Since $W$ is dense in $V_0$, there is some $y'\in V\cap W$. By Claim \ref{only one sheet}, $y'$ has exactly one preimage in $U_0$. By Claim \ref{good pair locally surjective}, $y'$ has at least one preimage in each of $U_1,...,U_l$. Since the $U_i$ are disjoint, this implies $l=1$, as desired.
\end{proof}
	By Claim \ref{conclusion of ramification prop}, the proof of Proposition \ref{ramification prop} is complete.
\end{proof}

\noindent\textit{Acknowledgments:} Thank you to Assaf Hasson for many helpful discussions of the material and for reading earlier drafts of the paper. Thanks also to Tom Scanlon, Rahim Moosa, Ya'acov Peterzil, and Dave Marker for helpful conversations, and to Chieu-Minh Tran for suggestions on terminology.

        \bibliography{references}
        \bibliographystyle{amsalpha}
    \end{document}